\documentclass{amsart}

\usepackage{
	amsmath,
	amssymb,
	amsthm,
	enumerate,
	longtable,
	mathptmx}

\usepackage[
	colorlinks=true,
	citecolor=black,
	linkcolor=black,
	urlcolor=black]{hyperref}

\usepackage{tikz}
\usetikzlibrary{calc,scopes,backgrounds,arrows}

\tikzset{%
	int/.style={fill=white,opacity=.85,pos=.5, inner sep=.1em,font=\scriptsize},%
	pair/.style={coordinate,shape=circle,draw,scale=.3,fill=white, text=white},%
	solid/.style={coordinate,shape=circle,draw,scale=.3,fill=black},%
	sqr/.style={scale=.4,fill=white,shape=square,outer sep = 3pt},%
	vert/.style={scale=.4,fill=black,shape=circle,outer sep = 3pt},%
	wert/.style={scale=.4,draw,fill=white,shape=circle,outer sep = 3pt},%
	outline/.style={draw,line width=1.5mm,white},%
	sel/.style={draw,shape=circle,scale=2,color=red},%
	R/.style={inner sep=1,fill=gray!20},%
	bcut/.style={blue,dashed,thick},%
	rcut/.style={red,thick},%
}
		
\pgfdeclarelayer{background}
\pgfsetlayers{background,main}

\DeclareMathOperator{\Cov}{Cov}
\DeclareMathOperator{\Hom}{Hom}
\DeclareMathOperator{\End}{End}
\DeclareMathOperator{\cmod}{\mathsf{mod}}
\DeclareMathOperator{\gr}{\mathsf{gr}}
\DeclareMathOperator{\add}{\mathsf{add}}

\DeclareMathOperator{\Homeo}{Homeo}

\DeclareMathOperator{\im}{Im}
\DeclareMathOperator{\MG}{\mathcal{M}\mathcal{G}}

\newcommand{\calc}{\mathcal{C}}
\newcommand{\cald}{\mathcal{D}}

\newcommand{\frakC}{\mathfrak{C}}
\newcommand{\frakD}{\mathfrak{D}}

\newcommand{\fieldstyle}[1]{\mathbb{#1}}
\renewcommand{\AA}{\fieldstyle{A}}

\newcommand{\NN}{\fieldstyle{N}}
                             
\newcommand{\RR}{\fieldstyle{R}}                             
\newcommand{\ZZ}{\fieldstyle{Z}}                             
\newcommand{\kk}{k}

\newcommand{\e}{\varepsilon}
\renewcommand{\tilde}{\widetilde}

\newcommand{\df}[1]{\emph{#1}}

\theoremstyle{plain}
\newtheorem{thm}{Theorem}
\newtheorem*{thm*}{Theorem}
\newtheorem{lemma}[thm]{Lemma}
\newtheorem{prop}[thm]{Proposition}
\newtheorem{cor}[thm]{Corollary}
\newtheorem*{cor*}{Corollary}
\theoremstyle{definition}
\newtheorem{dfn}[thm]{Definition}
\newtheorem{ex}[thm]{Example}
\newtheorem{remark}[thm]{Remark}

\newenvironment{enum}
	{\begin{enumerate} [\upshape (1)]}
	{\end{enumerate}}

\newenvironment{enuma}
	{\begin{enumerate} [\upshape (a)]}
	{\end{enumerate}}

\newcommand{\IF}{\text{if }}
\newcommand{\orientationofthearrow}{counter-clockwise }

\newcommand{\intri}{\mathfrak{IT}}

\newcommand{\dd}{\dagger}
\renewcommand{\dag}{^\dd}
\newcommand{\cut}[3]{\chi_{#1,#2,#3}}
\newcommand{\cutt}{\chi}
\newcommand{\alg}{\Lambda}
\newcommand{\bcpt}{B}
\newcommand{\te}[1]{\tau_{#1}}   
\newcommand{\pp}{\gamma}         

\begin{document}
\title{Derived Equivalence of Surface Algebras in Genus 0 via Graded Equivalence}
\author{Lucas David-Roesler}
\address{University of Connecticut, 196 Auditorium Road, Unit 3009 Storrs, CT 06269-3009} 
\thanks{The author was supported by the NSF grant DMS-1001637.}
\email{lucas.roesler@uconn.edu}
\thanks{The author would like to thank Professor Ralf Schiffler for his many valuable comments and suggestions in the preparation of this article.  In addition, discussions with Ben Salisbury proved very helpful. }
\date{\today}

\maketitle

\begin{abstract}
We determine some of the derived equivalences of a class of gentle algebras called
surface algebras. These algebras are constructed from an unpunctured Riemann
surface of genus 0 with boundary and marked points by introducing cuts in internal 
triangles of an arbitrary triangulation of the surface. In particular, we fix a
triangulation of a surface and determine when different cuts produce derived
equivalent algebras.
\end{abstract}

\section{Introduction}

Let $T$ be a triangulation of a bordered unpunctured Riemann surface $S$ with 
 a set of marked points $M$, and let $(Q_T,I_T)$ be the bound quiver associated to
$T$ as in \cite{ABCP,CCS}. The corresponding algebra $\alg_T=kQ_T/I_T$, over an
algebraically closed field $k$, is a finite-dimensional gentle algebra \cite{ABCP}
which is also the endomorphism algebra of the cluster-tilting object corresponding
to $T$ in the generalized cluster category associated to $(S,M)$, see
\cite{Amiot,BZ,BMRRT,CCS}. Each internal triangle in the triangulation $T$
corresponds to an oriented $3$-cycle in the quiver $Q_T$, and the relations for
the algebra $B_T$ state precisely that the composition of any two arrows in an
oriented $3$-cycle is zero in $\alg_T$.

In \cite{DS}, surface algebras were introduced as a new setting to describe the
iterated tilted algebras of Dynkin type $\AA$ and $\tilde\AA$, corresponding to
the case where $S$ is a disc and annulus respectively. In addition to the iterated
tilted algebras of type $\AA$ or $\tilde \AA$ with global dimension 2, the authors
obtained the larger class of surface algebras by realizing the concept of an
admissible cut, as defined in \cite{BFPPT}, in the surface. This procedure
increases the number of marked points in each boundary component while the number
of edges in the triangulation remains fixed, so the resulting algebra comes from a
partial triangulation of a surface. In terms of the quiver $Q_T$, we get a new
quiver associated to this partial triangulation by removing one arrow from the
oriented 3-cycles corresponding to internal triangles. The surface algebras that 
are not iterated tilted do not appear in any other known classification of algebras. 
In general, there are many different surface algebras that can arise even when we 
fix a triangulation. It is natural to ask how these new algebras are related to 
each other.

We focus on describing the derived categories of these algebras. This work is
motivated by the fact that derived equivalence in the disc and annulus is
relatively easy to check. For the surface algebras of the disc and annulus,
derived equivalence is determined by the derived invariant of Avella-Alaminos and
Geiss defined in \cite{AG}. This invariant is easy to calculate for surface
algebras, see \cite{DS}. However, for surfaces with higher genus or with more than
two boundary components, this invariant need not determine derived equivalence.
However, using the AG-invariant, we can show that there may be several derived
equivalence classes of algebras for a fixed triangulation of any surface
other than the disc, see \cite{DS}. In fact, for surface algebras from the
annulus, there must be at least two derived equivalence classes.

In this paper, we present a method for determining the derived equivalence of
surface algebras coming from a fixed triangulation of $T$ of a surface with genus
0. That is, we restrict ourselves to considering those surface algebras that come
from different cuts of the same triangulated surface. We do not attack this
directly, rather, we take advantage of recent work by Amiot and Oppermann
\cite{AO} in which they show that in certain cases derived equivalence is the same
as considering graded equivalence with respect to a suitable grading of the
arrows. In particular, this is true for surface algebras. This greatly simplifies
the problem because we are able to describe the graded equivalences in terms of
the cuts that define our algebras and automorphisms of the surface.

We denote cuts of a surface by $\chi$. A pair of two cuts $(\chi_1,\chi_2)$ is
called \df{equi-distributed} if for each boundary component $\bcpt$ of $S$, the
number of cuts in $\chi_1$ on $\bcpt$ is equal to the number of cuts in $\chi_2$
on $\bcpt$. When $\chi_1$ and $\chi_2$ are equi-distributed, we can view $\chi_1$
as being a permutation of $\chi_2$. Additionally, given a cut $\chi$ we get a
grading on $\alg_T$ by assigning the weight 1 to each arrow removed from $Q_T$ by
$\chi$ and 0 for all other arrows in $Q_T$, we denote the graded algebra obtained
in this way by $\tilde \alg$. We have our first main theorem.

\begin{thm*}
Let $(S,M,T)$ be a triangulated bordered surface of genus 0 and $\alg_1$ and
$\alg_2$ be surface algebras of type $(S,M,T)$ coming from admissible cuts
$\chi_1$ and $\chi_2$. Then $\tilde\alg_1$ and $\tilde\alg_2$ are graded
equivalent if there is an automorphism $f$ of the surface
(up to isotopy) such that $f$ induces a quiver automorphism on $Q_T$ 
and $(\chi_1,f(\chi_2))$ or $(f(\chi_1),\chi_2)$ are equi-distributed.
\end{thm*}

Using the work by Amiot and Oppermann, this theorem becomes 
a statement about derived equivalence. The graded equivalence of $\tilde\alg_1$
and $\tilde\alg_2$ becomes a derived equivalence of $\alg_1$ and $\alg_2$.

\begin{cor*}
Let $\alg_i$ and $\chi_i$ be as in the theorem. Then $\alg_1$ and $ \alg_2$ are
derived equivalent if there is an automorphism $f$ of the surface (up
to isotopy) such that $f$ induces a quiver automorphism on $Q_T$ and
$(\chi_1,f(\chi_2))$ or $(f(\chi_1),\chi_2)$ are equi-distributed.
\end{cor*}

Related work has been done for unpunctured surfaces without cuts.  Ladkani
\cite{Lad} uses quiver mutation to characterize the surfaces such that all the 
algebras arising from their  triangulations are derived equivalent. 
Bobinski and Buan \cite{BB} classified the gentle algebras that are derived 
equivalent to cluster-tilted algebras of type $\AA$ and $\tilde{\AA}$, these 
arise from the triangulations of the disc and annulus.  Their proof makes use 
Brenner-Butler tilting via reflections of gentle algebras. 
We realize a connection between these two methods of studying derived 
equivalence by characterizing the reflections of surface algebras as
cut versions of mutation in the surface. Let $R_x$ denote the reflection
of $Q$ at the vertex $x$ and $\mu_x$ the mutation at $x$.  We have the 
following theorem.
\begin{thm*} Let $(Q,I)$ be the quiver with relations of a surface algebra 
of type $(S,M,T)$. If $x$ is not the source of a relation in $(Q,I)$ and 
$R_x$ is defined, then there is an admissible cut of $\mu_x(Q_T)$ that 
gives $R_x(Q)$.
\end{thm*}

The use of reflections allows us to realize the derived equivalence of
surface algebras coming from \emph{different} triangulations. Additionally,
this theorem gives us a way to realize derived equivalences of surface 
algebras in the module category.  In the work of Amiot and Opperman \cite{AO},
they explicitly describe the tilting object associated to a graded equivalence.
This tilting object is specifically described in the derived category; hence, 
the derived equivalences given by non-trivial automorphisms of the surface are 
necessarily given by tilting objects in the derived category that can not be 
viewed as sitting in the module category.

We would like to remark that results of Amiot and Oppermann in \cite{AO2} give a
complete description of the derived equivalence classes of surface algebras of
type $\tilde\AA$. They do this by considering graded mutations of quivers with
potentials and introducing an invariant called the weight or algebra.  Similar 
work for the surface algebras of type $\tilde\AA$ was also done in \cite{DS} 
using the AG-invariant. Using reflections, we give an alternative realization
of the derived equivalences of surface algebras of type $\tilde\AA$.

In Sections 2 and 3 we introduce the necessary definitions and background 
about surface algebras and graded algebras.  Section 3 ends with a partial
description of the graded equivalences given by the identity map on $S$.
Section 4 contains the main theorem of the paper, extending the description
in Section 3 to other elements in the mapping class group of $(S,M)$. Note 
that the definition of the mapping class group is different from the usual 
definition.  Section 5 reformulates the theorems about graded equivalences 
in terms of derived equivalences. The final section considers derived 
equivalences of surface algebras given by reflections of gentle algebras.  

\section{Preliminaries and Notation}
In this section we give an alternative but equivalent definition of surface algebras from 
\cite{DS}.

\subsection{Triangulated surfaces}
Let $S$ be a connected oriented unpunctured Riemann surface with boundary
$\partial S$ and let $M$ be a non-empty finite subset of the boundary $\partial
S$ with at least one point in each boundary component. The elements of $M$ are 
called \df{marked points}. We will refer to the pair $(S,M)$ simply as an 
\df{unpunctured surface}.

We say that two curves in $S$ \emph{do not cross} if they do not intersect
each other except that the endpoints may coincide.

\begin{dfn}\label{def arc}
An \emph{arc} $\pp$ in $(S,M)$ is a curve in $S$ such that 
\begin{enuma}
\item the endpoints are in $M$,
\item $\pp$ does not cross itself,
\item the relative interior of $\pp$ is disjoint from $M$ and
  from the boundary of $S$,
\item $\pp$ does not cut out a monogon or a digon. 
\end{enuma}
If $\pp$ is called a \df{generalized arc} if it satisfies only conditions (a), (c) 
and (d).
\end{dfn}

The \df{boundary segments} of $S$ are those curves that connect two
marked points and lie entirely on the boundary of $S$ without passing
through a third marked point

We consider generalized arcs up to isotopy inside the class of such curves.
Moreover, each generalized arc is considered up to orientation, so if a
generalized arc has endpoints $a,b\in M$ then it can be represented by a curve
that runs from $a$ to $b$, as well as by a curve that runs from $b$ to $a$.

For any two arcs $\pp,\pp'$ in $S$, let $e(\pp,\pp')$ be the minimal
number of crossings of $\pp$ and $\pp'$, that is, $e(\pp,\pp')$ is
the minimum of the numbers of crossings of curves $\alpha$ and $\alpha'$, where
$\alpha$ is isotopic to $\pp$ and $\alpha'$ is isotopic to $\pp'$. Two
arcs $\pp,\pp'$ are called \emph{non-crossing} if $e(\pp,\pp')=0$. A
\emph{triangulation} is a maximal collection of non-crossing arcs. The arcs of a
triangulation cut the surface into \emph{triangles}. Since $(S,M)$ is an
unpunctured surface, the three sides of each triangle are distinct (in contrast
to the case of surfaces with punctures). A triangle in $T$ is called an
\emph{internal triangle} if none of its sides are a boundary segment. We often
refer to the triple $(S,M,T)$ as a \emph{triangulated surface}.  

\subsection{Jacobian algebras from surfaces}
Let $Q=(Q_0,Q_1,s,t)$ be a quiver with vertex set $Q_0$, $Q_1$ the arrow set,
and $s,t\colon Q_1\to Q_0$ are maps that assign to each arrow $\alpha$ its
source $s(\alpha)$ and target $t(\alpha)$. For $v,v'\in Q_0$, we let $Q_1(v,v')$
denote the set of arrows from $v$ to $v'$.

If $T=\{\te1,\te2,\ldots,\te n\}$ is a triangulation of an unpunctured
surface $(S,M)$, we define a quiver $Q_T$ as follows. Each arc in $T$
corresponds to a vertex of $Q_T$. We will denote the vertex corresponding to
$\te i$ simply by $i$. The number of arrows from $i$ to $j$ is the number of
triangles $\triangle$ in $T$ such that the arcs $\te i,\te j$ form two sides
of $\triangle$, with $\te j$ following $\te i$ when going around the triangle
$\triangle$ in the \orientationofthearrow orientation, see 
Figure~\ref{fig quiver} for an example. For clarity we suppress the $\te{}$ notation 
when there is no possibility of confusion.  Note that the interior triangles in $T$ 
correspond to certain oriented 3-cycles in $Q_T$. 
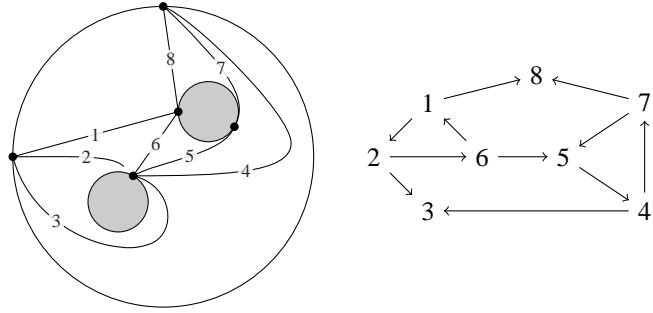
\begin{figure}
\centering

\begin{tikzpicture}[scale=.4]
{ []
\draw (0,0) circle[radius=5cm] ;
\filldraw[fill=black!20] (1.5,1.5) circle[radius=1cm] ;
\filldraw[fill=black!20] (-1.5,-1.5) circle[radius=1cm] ;

\foreach \x [count=\n] in {0,1,...,11}{
	\node[name=p\n] at ($(1.5,1.5)+(30*\x:1cm)$) {};
	\node[name=n\n] at ($(-1.5,-1.5)+(30*\x:1cm)$) {};
	\node[name=o\n] at (30*\x:5cm) {};
}

\node[solid] at (90:5cm) {};
\node[solid] at (180:5cm) {};
\node[solid] at ($(-1.5,-1.5)+(60:1cm)$) {};
\node[solid] at ($(1.5,1.5)+(180:1cm)$) {};
\node[solid] at ($(1.5,1.5)+(-30:1cm)$) {};

\draw 
  (o7.center) -- (p7.center) 
    node[int] {1} 
  (o7.center) ..controls +(0:2cm) and +(135:1cm) .. (n3) 
    node[int] {2} 
  (o7.center) .. controls +(-75:2cm) and +(170:1cm) .. ($(o9)!.75!(n10)$) 
    node[int] {3}
  ($(o9)!.75!(n10)$) .. controls +(-10:.5cm) and +(200:.5cm) .. 
    ($(n12)!.3!(o10)$) 
  ($(n12)!.3!(o10)$) .. controls +(20:1cm) and (-70:1cm) .. (n3.center)
  (p12.center) .. controls +(60:1cm) and +(-60:.5cm) .. ($(p4)!.33!(o3)$) 
    node[int,pos=.9] {7}
  ($(p4)!.33!(o3)$) .. controls +(130:.cm) and +(-40:1cm) .. (o4.center) 
  (n3.center) .. controls +(0:3cm) and +(-90:1cm) .. ($(p1)!.7!(o1)$) 
    node[int] {4}
  ($(p1)!.7!(o1)$) .. controls +(90:1cm) and +(-20:1cm) .. (o4.center) 
  (o4.center) -- (p7.center) node[int] {8} 
  (p7.center) -- (n3.center)  node[int] {6}
  (p12.center) .. controls +(-110:1cm) and +(30:1cm) .. (n3.center)  
    node[int] {5}; 
}
{[xshift=7cm,scale=1.8]
	\node[name=1] at (1,1) {1};
	\node[name=2] at (0,0) {2};
	\node[name=3] at (1,-1) {3};
	\node[name=4] at (5,-1) {4};
	\node[name=5] at (3.5,0) {5};
	\node[name=6] at (2,0) {6};
	\node[name=7] at (5,1) {7};
	\node[name=8] at (3,1.5) {8};
	\draw[->] (1) edge (2) edge (8)
		(2) edge (6) edge (3)
		(6) edge (1) edge (5) 
		(5) edge (4)
		(4) edge (7) edge (3)
		(7) edge (5) edge (8);
}
\end{tikzpicture}
\caption{A triangulation and its quiver} \label{fig quiver}
\end{figure}   

Following \cite{ABCP,LF}, let $W$ be the sum of all oriented 3-cycles in $Q_T$
coming from internal triangles. Then $W$ is a potential, in the sense of
\cite{DWZ}, which gives rise to to a Jacobian algebra $\alg_T=
\textup{Jac}(Q_T,W)$, which is defined as the quotient of the path algebra of
the quiver $Q_T$ by the two-sided ideal $I_T$ generated by the subpaths of length two
of each oriented 3-cycle in $Q_T$.

\subsection{Cutting a surface}

Let $(S,M)$ be a surface without punctures, $T$ a triangulation, $Q_T$ the
corresponding quiver, and $\alg_T$ the Jacobian algebra. Throughout this
section, we assume that, if $S$ is a disc, then $M$ has at least $5$ marked
points, thus we exclude the disc with $4$ marked points.  
\begin{dfn}
Recall that the interior triangles of $T$ distinguish certain oriented
$3$-cycles in the quiver $Q_T$.  Let $\intri$ denote the set of 
internal triangles of $(S,M,T)$.  We define an \df{admissible cut} of $T$ 
to be a function $\cutt\colon \intri \to M$ that selects a vertex in each
internal triangle of $T$. 
\end{dfn}

In addition to selecting a marked point $\cutt(\triangle)=v$ 
on the surface, this map also distinguishes the two edges $\te i$ and $\te j$ in $T$ 
incident to $v$ in $\triangle$.  We call the image of $\cutt$ in $\triangle$ a \df{local cut}
of $(S,M,T)$, denoted $\cut vij$ or $\cutt_{i,j}$ when there is no cause for confusion.  
We will always write $\cutt_{i,j}$ when the corresponding arrow is $i\to j$ in $Q_T$. 
Graphically, we will denote a local cut in $(S,M,T)$ by bisecting the marked point 
$\cutt(\triangle)$ between the corresponding edges $\te i$ and $\te j$, see 
Figure~\ref{fig alt presentation}.  The decorated surface corresponding to $\cutt$ is 
denoted $(S,M\dag,T\dag)$. 

\begin{dfn}
Note that a local cut of $(S,M,T)$ distinguishes an arrow in the quiver $Q_T$ associated
to the triangulation. Let $\cutt$ be an admissible cut of $(S,M,)$.  By an abuse of notation, 
let $\cutt$ also denote this collection of arrows, then we define the \df{surface algebra} 
$\alg_{T\dag}$ of type $(S,M)$ associated to $\cutt$ to be the quotient 
$\kk Q_T/\langle I_T \cup \cutt\rangle$, we let $I_{T\dag}$ denote the corresponding ideal
of relations on $Q_{T\dag}$.
\end{dfn}

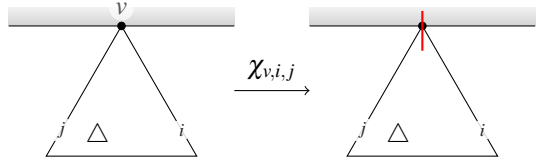
\begin{figure}
\centering

\begin{tikzpicture}
 [lbl/.style={fill=white,opacity=.66,shape=circle,inner sep=1,outer sep =2}]
{[]
\shade[shading=axis,bottom color=black!20, top color=black!5]  
    (0,0) rectangle (3,.25);
\draw (0,0) -- (3,0) node[solid,pos=.5,scale=1,name=v] {};
\node[above,lbl] at (v) {$v$};
\draw (v.center) -- ++(-60:2cm) node[int,pos=.8] {$i$}
	-- ++(180:2cm) node[above,pos=.66] {$\triangle$}
	-- (v.center) node[int,pos=.2] {$j$};
}
\draw[->] (3,-.86) -- (4,-.86) node[above,pos=.5] {$\cut{v}{i}{j}$};
{[xshift=4cm]
\shade[shading=axis,bottom color=black!20, top color=black!5]  
    (0,0) rectangle (3,.25);
\draw (0,0) -- (3,0) node[solid,pos=.5,scale=1,name=v] {};
\draw (v.center) -- ++(-60:2cm) node[int,pos=.8] {$i$}
	-- ++(180:2cm) node[above,pos=.66] {$\triangle$}
	-- (v.center) node[int,pos=.2] {$j$};
\draw[thick,red] ($(v)+(90:.2)$) -- ($(v)+(270:.33)$);
}
\end{tikzpicture}
\caption{The graphical notation for the $\cut vij$.}
\label{fig alt presentation}
\end{figure}

\begin{ex}
Here we present an admissible cut of the surface $(S,M,T)$ in Figure~\ref{fig quiver}
and the associated quiver of $\alg_{T\dag}$. 

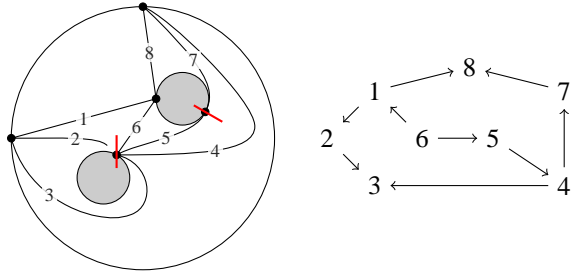
\[\begin{tikzpicture}[scale=.35]
{ []
\draw (0,0) circle[radius=5cm] ;
\filldraw[fill=black!20] (1.5,1.5) circle[radius=1cm] ;
\filldraw[fill=black!20] (-1.5,-1.5) circle[radius=1cm] ;

\foreach \x [count=\n] in {0,1,...,11}{
	\node[name=p\n] at ($(1.5,1.5)+(30*\x:1cm)$) {};
	\node[name=n\n] at ($(-1.5,-1.5)+(30*\x:1cm)$) {};
	\node[name=o\n] at (30*\x:5cm) {};
}

\node[solid] at (90:5cm) {};
\node[solid] at (180:5cm) {};
\node[solid] at ($(-1.5,-1.5)+(60:1cm)$) {};
\node[solid] at ($(1.5,1.5)+(180:1cm)$) {};
\node[solid] at ($(1.5,1.5)+(-30:1cm)$) {};

\draw 
  (o7.center) -- (p7.center) 
    node[int] {1} 
  (o7.center) ..controls +(0:2cm) and +(135:1cm) .. (n3) 
    node[int] {2} 
  (o7.center) .. controls +(-75:2cm) and +(170:1cm) .. ($(o9)!.75!(n10)$) 
    node[int] {3}
  ($(o9)!.75!(n10)$) .. controls +(-10:.5cm) and +(200:.5cm) .. 
    ($(n12)!.3!(o10)$) 
  ($(n12)!.3!(o10)$) .. controls +(20:1cm) and (-70:1cm) .. (n3.center)
  (p12.center) .. controls +(60:1cm) and +(-60:.5cm) .. ($(p4)!.33!(o3)$) 
    node[int,pos=.9] {7}
  ($(p4)!.33!(o3)$) .. controls +(130:.cm) and +(-40:1cm) .. (o4.center) 
  (n3.center) .. controls +(0:3cm) and +(-90:1cm) .. ($(p1)!.7!(o1)$) 
    node[int] {4}
  ($(p1)!.7!(o1)$) .. controls +(90:1cm) and +(-20:1cm) .. (o4.center) 
  (o4.center) -- (p7.center) node[int] {8} 
  (p7.center) -- (n3.center)  node[int] {6}
  (p12.center) .. controls +(-110:1cm) and +(30:1cm) .. (n3.center)  
    node[int] {5}; 

\draw[rcut]
  ($(n3)+(90:.75cm)$) -- ($(n3)-(90:.5cm)$)
  ($(p12)+(330:.75cm)$) -- ($(p12)-(330:.5cm)$);
  
}
{[xshift=7cm,scale=1.8]
	\node[name=1] at (1,1) {1};
	\node[name=2] at (0,0) {2};
	\node[name=3] at (1,-1) {3};
	\node[name=4] at (5,-1) {4};
	\node[name=5] at (3.5,0) {5};
	\node[name=6] at (2,0) {6};
	\node[name=7] at (5,1) {7};
	\node[name=8] at (3,1.5) {8};
	\draw[->] (1) edge (2) edge (8)
		(2)  edge (3)
		(6) edge (1) edge (5) 
		(5) edge (4)
		(4) edge (7) edge (3)
		(7) edge (8);
}
\end{tikzpicture}\]
\end{ex}

See \cite{DS} for a complete description of surface algebras arising from admissible 
cuts in terms of partially triangulated surfaces and partial cluster-tilting objects.

\section{Graded Equivalence}
Ultimately, we are interested in describing the derived equivalence classes of
surface algebras. To this end, we are led to investigate graded equivalences of
graded algebras because of a theorem of Amiot and Oppermann in \cite[Theorem
5.6]{AO} showing a strong connection between the two types of equivalences.

In this section we introduce the concept of graded equivalence and seek to give
our first criteria for graded equivalence of surface algebras.

\subsection{Graded algebras}
We will only consider $\ZZ$-graded algebras, however, the following definitions
can be re-stated for any group $G$, as in \cite{GM}.  We will simply refer to 
$\ZZ$-gradings as gradings.

A \df{weight function} on $Q$ is a function $w\colon Q_1\to \ZZ$, that is, a
function that assigns an integer to each arrow of $Q$. We can naturally extend
the weight function to paths in $Q$, by setting $w(e_i)=0$ for each stationary
path in $Q$ and $w(\alpha_1\cdots\alpha_r)= w(\alpha_1)+\cdots+w(\alpha_r)$ for
each path in $Q$ with length $r\geq 1$. This induces a grading on $kQ$
with $kQ=\bigoplus_{p\in \ZZ} kQ^p$, where $kQ^p$ is generated by the set of
paths with weight $p$. A relation $r$ is homogeneous of degree $p$ if 
$r\in kQ^p$ for some $p$. The grading induced by $w$ gives a grading on $kQ/I$ 
if and only if $I$ is generated by homogeneous relations, not necessarily all 
of the same degree.

Let $\alg = \bigoplus_{p\in\ZZ} \alg^p$ be a graded algebra. As in
\cite{GM}, we denote by $\gr \alg$ the category of finitely generated graded
modules over $\alg$. For a graded module $M=\bigoplus_{p\in\ZZ} M^p$, we define
$M\langle q\rangle := \bigoplus_{p\in\ZZ} M^{p+q}$. That is, the $p$ graded part
of $M\langle q\rangle$ is the $p+q$ graded part of $M$. 

We use this grading shift to define a new category that will, in some ways, take
on the role of the derived category. Of course, this new category is relatively
simpler. 

\begin{dfn} Given a graded algebra $\alg=kQ/I$ induced by a weight $w$, we
define the covering of $\alg$
  \[\Cov(\alg):= \add\left\{ \alg\langle p\rangle : p\in \ZZ\right\} 
      \subseteq \gr \alg.\]
Let $F\colon \gr \alg \to \cmod \alg$ be the functor that forgets the grading.
We associate to $\Cov(\alg)$ the quiver with relations $(Q^*,I^*)$ defined by  
\begin{align*} 
	Q^*_0 &= Q_0\times \ZZ,\\
	Q^*_1((v,i),(v',j)) &= \{ \alpha\in Q_1(v,v') : w(\alpha) = j-i\}. 
\end{align*}
Note that $Q^*$ is infinite. The map $F$ induces a projection $Q^*\to Q$, we 
will also refer to this as $F$. We define the relations on $Q^*$ by 
$\rho \in I^*$ if $F(\rho)\in I$.  We partition the vertices of $Q^*$ into
\df{levels} where $(v,i)$ is of level $i$. If $w(\alpha)>0$, we refer to 
the copies of $\alpha$ in $Q^*$ as \df{bridge arrows}, these arrows connect
different levels of $Q^*$. 
\end{dfn}

From \cite[Theorem 0.1]{GM} we have,
\begin{prop} Let $\alg$ be a finite dimensional graded algebra and $(Q,I)$ a 
quiver with relations and weight $w$ such that $\alg \cong kQ/I$ and the 
grading on $\alg$ is induced by $w$, then
$\cmod kQ^*/I^* \cong \cmod \Cov(\alg) \cong \gr \alg$.
\end{prop}

Additionally, we recall from \cite[Theorem 2.11]{AO},
\begin{prop}\label{prop:graded equiv}
Let $\alg$ be an algebra with two different gradings.  We denote by
$\Cov(\alg_1)$ the covering corresponding to the first grading, and 
$\Cov(\alg_2)$ the covering corresponding to the second grading. 
Then the following are equivalent:
\begin{enum}
\item There is an equivalence $U\colon\cmod\Cov(\alg_1) 
\xrightarrow{\ \sim \ } \cmod\Cov(\alg_2)$ such that the following 
diagram commutes.

\[\begin{tikzpicture}[xscale=2]
\node[name=a] at (0,0) {$\cmod\Cov(\alg_1)$};
\node[name=b] at (3,0) {$\cmod\Cov(\alg_2)$};
\node[name=c] at (1.5,-1) {$\cmod \alg$};
\path[->] (a) edge node[above]{$U$} (b) (a) edge (c) (b) edge (c);
\end{tikzpicture}\]
\item There exist a map $r\colon Q_0\to \ZZ$ with $r(i) = r_i$ and an isomorphism 
of graded algebras
\[f\colon\alg_2 \xrightarrow{\ \sim \ }\bigoplus_{p\in \ZZ}
    \Hom_{\Cov(\alg_1)}\left(\bigoplus_{i=1}^n P_i\langle r_i\rangle,
                          \bigoplus_{i=1}^n P_i\langle r_i+p\rangle\right)
\]
where $\alg_1\cong \bigoplus_{i=1}^n P_i$ in $\gr\alg_1$.
\end{enum}
In this case we say that the gradings are equivalent.
\end{prop}

\begin{remark}\label{remark:equivviaid}
The isomorphism $f\colon\alg_2 \xrightarrow{\ \sim \ }\bigoplus_{p\in \ZZ}
    \Hom_{\Cov(\alg_1)}\left(\bigoplus_{i=1}^n P_i\langle r_i\rangle,
                          \bigoplus_{i=1}^n P_i\langle r_i+p\rangle\right)$
may arise by first applying a automorphism to the quiver $Q^*$ of $\Cov(\alg_1)$.
The simplest case to consider is when this automorphism is the identity on $Q^*$,
when this happens, the graded equivalence can
be checked via purely combinatorial methods involving the quiver $Q^*$
associated to $\Cov(\alg_1)$. In particular, let $w_2$ be the weight determined
by the grading of $\alg_2$, then verifying (2) reduces to finding vertices
$(v,i)$ and $(v',j)$ of $Q^*$ such that if $\alpha\colon v\to v'$ and
$w_2(\alpha)=k$, there is an arrow $(v,i) \to (v',j+k)$ in $Q^*$.  Then we can 
define the map $r$ such that $r(v) = i$ and $r(v') = j$. We will use this fact
in the proof of the main theorem. The algebras $\alg_1$ and $\alg_2$ are graded
equivalent if such a choice can be made simultaneously for each vertex. For
brevity we will later refer to this as being \df{graded equivalent via the
identity}.  Notice that we must have $j>i$, because $\alpha$ is a bridge arrow, 
which by definition must always point in an increasing direction.
See example~\ref{ex:pants}.
\end{remark}

We will \emph{not} consider the surfaces algebras as graded algebras.  However, 
the cut defining a surface algebra does induce a grading on the algebra coming 
from the original triangulation.
\begin{dfn}\label{dfn:weight}
Let $\alg$ be a surface algebra coming from an admissible cut of $(S,M,T)$.
Let $\tilde \alg$ denote the Jacobian algebra coming from $(Q_T,W)$ with a
grading given by the weight
\[ w(\alpha) = \begin{cases} 
                0 & \text{ if } \alpha\in Q_T\cap Q_T\dag,\\
                1 & \text{ if } \alpha\in Q_T\setminus Q_T\dag.
                \end{cases} \]
This weight is homogeneous for all relations in $(Q_T,W)$, hence
it induces a grading on $\tilde \alg$.
\end{dfn}

\subsection{Graded equivalence and surface algebras}\label{subsec:gradequiv}
In this section we describe when two surface algebras are graded equivalent via
the identity.  To that end we begin by finding the required integers $r_i$, as 
in Proposition~\ref{prop:graded equiv} and Remark~\ref{remark:configurations}, 
for those vertices corresponding to edges in $(S,M,T)$ incident to a cut.  
Throughout we fix two different admissible cuts $\chi_1$ and $\chi_2$ of 
$(S,M,T)$ with $Q_i$ the corresponding cut quivers, $\alg_i$ the corresponding 
surface algebras, $\tilde\alg_i$ the corresponding graded Jacobian algebras, and 
$Q^*$ the quiver of $\Cov(\tilde \alg_1)$.

\begin{dfn}
Given a pair of cuts $(\chi_1,\chi_2)$ let $\{\te{i_1},\dots,\te{i_k}\}$ be the 
set of edges in $(S,M,T)$ such that $\te{i_\ell}$ is the edge of a triangle in 
which $\chi_1$ and $\chi_2$ differ and $\te{i_\ell}$ is incident to both cuts. 
We call the edges in $\{\te{i_1},\dots,\te{i_k}\}$ \df{sliding edges}. Notice 
that there is at most one sliding edge for each internal triangle of $(S,M,T)$. 
Additionally, each sliding edge is associated with at least one internal triangle; 
however, there may be sliding edges $\te i$ associated with two different 
triangles. When necessary we may distinguish between the different types of 
sliding edges as \df{one-sliding} and \df{two-sliding} edges, respectively.
\end{dfn}

\begin{remark}\label{remark:configurations}
Recall that the local cut $\cutt_{i,j}$ denotes the cut which removes the arrow
$i\to j$. Let $\tilde \alg_1\sim \tilde \alg_2$ be graded equivalent via the
identity. By considering the orientation of the arrows which are cut and
definition~\ref{dfn:weight} of the weight given by a cut, we give an explicit
formula for determining the function $r$ from Proposition~\ref{prop:graded
equiv} (2) on triangles containing sliding edges. Since the weight of an arrow
is at most 1, the value of $r$ can only differ by one near sliding edges. We
first consider triangles where $\te i$ is a two-sliding edge, so there are
internal triangles $\triangle = \te i\te j \te k$ and $\triangle'=\te i\te j'\te
k'$. For $\te i$ to be a two sliding edge, when we restrict to $\triangle$ and
$\triangle'$, we must have
\[ 
(\text{a})\ \chi_1=\chi_{ki}\chi_{ij'}\text{ and }\chi_2=\chi_{ij}\chi_{k'i} \quad\text{or}
  \quad (\text{b})\ \chi_1=\chi_{ij}\chi_{ij'}\text{ and }\chi_2=\chi_{ki}\chi_{k'i},
\]
see Figure~\ref{fig:proof2sliding}. If we let $r(i)$ be any integer, then a
graded equivalence via the identity implies that we must have 
$r(\ell) = r(i)+1$ for $\ell=j',k'$ and $r(\ell) = r(i)-1$ for $\ell=j,k$
in the first case, in the second case we must have 
$r(\ell)=r(i)+1$ for $\ell = j,j',k,k'$. 
In both cases, the full subquiver on the $P_\ell\langle r_\ell\rangle$ in 
$Q^*$ contains the bridge arrows associated to $\chi_1$.

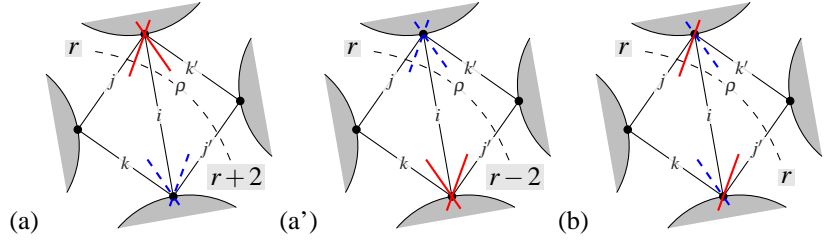
\begin{figure}[h]
\centering
(a)
\begin{tikzpicture}[scale=.8]
\foreach \n/\a in {a/0,b/90,c/180,d/270}{
\filldraw[fill=gray!50,rotate=\a+10] 
    (0+30:2) .. controls +(230:.75) and +(130:.75) .. (0-30:2)
    node[solid,name=\n,pos=.5] {};    
}
\draw (b) -- (d) node[int] {$i$}
      (b) -- (c) node[int] {$j$}
      (c) -- (d) node[int] {$k$}
      (a) -- (b) node[int] {$k'$}
      (a) -- (d) node[int] {$j'$};
\draw[red,thick] 
	  ($(b)+(70:.25)$) -- ($(b)+(250:.75)$)
	  ($(b)+(125:.25)$) -- ($(b)+(305:.75)$);
\draw[blue,thick,dashed]
	  ($(d)+(70:.75)$) -- ($(d)+(250:.25)$)
	  ($(d)+(125:.75)$) -- ($(d)+(305:.25)$);
\path (135:1.5) edge[bend left=30,dashed] node[int] {$\rho$} 
      node[R,pos=-.1] {$r$} 
      node[R,pos=1.05] {$r+2$} (325:1.5);
\end{tikzpicture}
(a')
\begin{tikzpicture}[scale=.8]
\foreach \n/\a in {a/0,b/90,c/180,d/270}{
\filldraw[fill=gray!50,rotate=\a+10] 
    (0+30:2) .. controls +(230:.75) and +(130:.75) .. (0-30:2)
    node[solid,name=\n,pos=.5] {};    
}
\draw (b) -- (d) node[int] {$i$}
      (b) -- (c) node[int] {$j$}
      (c) -- (d) node[int] {$k$}
      (a) -- (b) node[int] {$k'$}
      (a) -- (d) node[int] {$j'$};
\draw[blue,thick,dashed] 
	  ($(b)+(70:.25)$) -- ($(b)+(250:.75)$)
	  ($(b)+(125:.25)$) -- ($(b)+(305:.75)$);
\draw[red,thick]
	  ($(d)+(70:.75)$) -- ($(d)+(250:.25)$)
	  ($(d)+(125:.75)$) -- ($(d)+(305:.25)$);
\path (135:1.5) edge[bend left=30,dashed] node[int] {$\rho$} 
      node[R,pos=-.1] {$r$} 
      node[R,pos=1.05] {$r-2$} (325:1.5);
\end{tikzpicture}
(b)
\begin{tikzpicture}[scale=.8]
\foreach \n/\a in {a/0,b/90,c/180,d/270}{
\filldraw[fill=gray!50,rotate=\a+10] 
    (0+30:2) .. controls +(230:.75) and +(130:.75) .. (0-30:2)
    node[solid,name=\n,pos=.5] {};    
}
\draw (b) -- (d) node[int] {$i$}
      (b) -- (c) node[int] {$j$}
      (c) -- (d) node[int] {$k$}
      (a) -- (b) node[int] {$k'$}
      (a) -- (d) node[int] {$j'$};
\draw[red,thick] 
	  ($(b)+(70:.25)$) -- ($(b)+(250:.75)$)
	  ($(d)+(70:.75)$) -- ($(d)+(250:.25)$);
\draw[blue,thick,dashed]
	  ($(b)+(125:.25)$) -- ($(b)+(305:.75)$)
	  ($(d)+(125:.75)$) -- ($(d)+(305:.25)$);
\path (135:1.5) edge[bend left=30,dashed] node[int] {$\rho$} 
      node[R,pos=-.1] {$r$} 
      node[R,pos=1.05] {$r$} (325:1.5);
\end{tikzpicture}
\caption{Configurations for two-sliding edges and the corresponding
 choices for $r_i$.  The solid red line represents $\chi_1$, the dashed 
 blue line for $\chi_2$. The dual configuration of ($b$) does not change the value
 of $r$.}\label{fig:proof2sliding}
\end{figure}

Now we consider the triangles $\te i\te j \te k$ where $\te i$ is an one-sliding edge. 
Then we must have 
\[
(\text{a})\ \chi_1=\chi_{ij} \text{ and } \chi_2=\chi_{ki},\quad\text{or}\quad
(\text{b})\ \chi_1=\chi_{ki} \text{ and } \chi_2=\chi_{ij},
\]
see Figure~\ref{fig:proof1sliding}.
If we let $r(i)$ be any integer, then in the first case we must choose $r(\ell) =
r(i)+1$ for $\ell=j,k$. In the second case, $r(\ell)=r(i)-1$ for $\ell = j,k$.
Again, in both cases the full subquiver on the $P_\ell\langle r_\ell\rangle$
contains the bridge arrow associated to $\chi_1$.

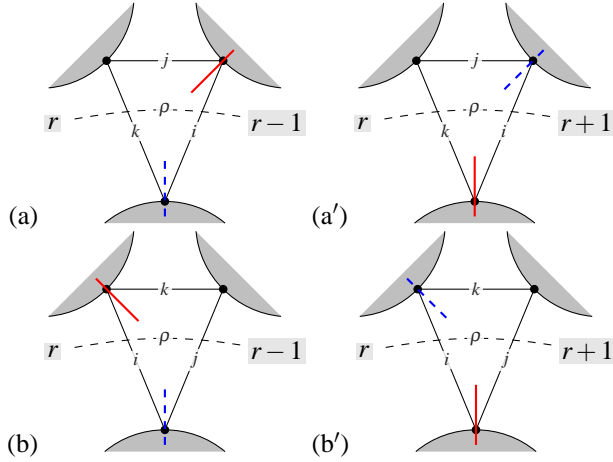
\begin{figure}[h]
\centering
(a)
\begin{tikzpicture}[scale=.8]
\foreach \n/\a in {a/45,b/135,c/270}{
\filldraw[fill=gray!50,rotate=\a] 
    (0+30:2) .. controls +(230:.75) and +(130:.75) .. (0-30:2)
    node[solid,name=\n,pos=.5] {};    
}
\draw (a) -- (c) node[int] {$i$}
      (a) -- (b) node[int] {$j$}
      (b) -- (c) node[int] {$k$};
\draw[blue,thick,dashed] 
	  ($(c)+(270:.25)$) -- ($(c)+(90:.75)$);
\draw[red,thick]
	  ($(a)+(45:.25)$) -- ($(a)+(225:.75)$);
\path[dashed] (-1.5,0) edge[bend left=10] node[int] {$\rho$}
              node[R,pos=-.1] {$r$}
              node[R,pos=1.1] {$r-1$} (1.5,0) ;
\end{tikzpicture}
(a$'$)
\begin{tikzpicture}[scale=.8]
\foreach \n/\a in {a/45,b/135,c/270}{
\filldraw[fill=gray!50,rotate=\a] 
    (0+30:2) .. controls +(230:.75) and +(130:.75) .. (0-30:2)
    node[solid,name=\n,pos=.5] {};    
}
\draw (a) -- (c) node[int] {$i$}
      (a) -- (b) node[int] {$j$}
      (b) -- (c) node[int] {$k$};
\draw[red,thick] 
	  ($(c)+(270:.25)$) -- ($(c)+(90:.75)$);
\draw[blue,thick,dashed]
	  ($(a)+(45:.25)$) -- ($(a)+(225:.75)$);
\path[dashed] (-1.5,0) edge[bend left=10] node[int] {$\rho$}
              node[R,pos=-.1] {$r$}
              node[R,pos=1.1] {$r+1$} (1.5,0) ;
\end{tikzpicture}

(b)
\begin{tikzpicture}[scale=.8]
\foreach \n/\a in {a/45,b/135,c/270}{
\filldraw[fill=gray!50,rotate=\a] 
    (0+30:2) .. controls +(230:.75) and +(130:.75) .. (0-30:2)
    node[solid,name=\n,pos=.5] {};    
}
\draw (a) -- (c) node[int] {$j$}
      (a) -- (b) node[int] {$k$}
      (b) -- (c) node[int] {$i$};
\draw[blue,thick,dashed] 
	  ($(c)+(270:.25)$) -- ($(c)+(90:.75)$);
\draw[red,thick]
	  ($(b)+(135:.25)$) -- ($(b)+(315:.75)$);
\path[dashed] (-1.5,0) edge[bend left=10] node[int] {$\rho$}
              node[R,pos=-.1] {$r$}
              node[R,pos=1.1] {$r-1$} (1.5,0) ;
\end{tikzpicture}
(b$'$)
\begin{tikzpicture}[scale=.8]
\foreach \n/\a in {a/45,b/135,c/270}{
\filldraw[fill=gray!50,rotate=\a] 
    (0+30:2) .. controls +(230:.75) and +(130:.75) .. (0-30:2)
    node[solid,name=\n,pos=.5] {};    
}
\draw (a) -- (c) node[int] {$j$}
      (a) -- (b) node[int] {$k$}
      (b) -- (c) node[int] {$i$};
\draw[red,thick] 
	  ($(c)+(270:.25)$) -- ($(c)+(90:.75)$);
\draw[blue,thick,dashed]
	  ($(b)+(135:.25)$) -- ($(b)+(315:.75)$);
\path[dashed] (-1.5,0) edge[bend left=10] node[int] {$\rho$}
              node[R,pos=-.1] {$r$}
              node[R,pos=1.1] {$r+1$} (1.5,0) ;
\end{tikzpicture}
\caption{Configurations for one-sliding edges and the corresponding choices
for $r_i$. The solid red line represents $\chi_1$, the dashed blue line for 
$\chi_2$.}
\label{fig:proof1sliding}
\end{figure}
\end{remark}

It remains to determine the appropriate value of $r$ for the non-sliding edges 
not contained in a triangle with a sliding edge.  Recall that $Q^*$ consists of 
infinitely many copies of $Q_1$ connected by arrows $i\to j$ for each local cut 
$\chi_{ij}$ in $\chi$, we refer to each copy of $Q_1$ as a \df{level} of $Q^*$. 
If $\tilde\alg_1\sim\tilde\alg_2$, we refer to the vertices $(v,\ell)$ such that 
$r(v)=\ell$ as the $\ell$-th \df{level partition}. In example~\ref{ex:pants}, the 
level partitions are the circled vertices of a particular level. 

\begin{ex}\label{ex:pants}Let $(S,M,T)$ be the surface given in Figure~\ref{fig:SurfaceEx1}.
If we consider the cuts
\begin{align*}
	\chi_1 &= \chi_{9,8}\chi_{4,2 }\chi_{11,2 }\chi_{12,1 }
 					 \chi_{5,7}\chi_{13,3}\chi_{12,14}\chi_{16,17},\\
	\chi_2 &= \chi_{9,8}\chi_{4,2}\chi_{ 2, 1}\chi_{3,12 }
 					 \chi_{5,7}\chi_{ 6,13}\chi_{12,14}\chi_{17,18},
\end{align*}

and $\alg_i$ given by $\chi_i$. The quiver of $\Cov(\tilde \alg_1)$ is 
given in Figure~\ref{fig:cov1}. Letting $P_i\langle r_i\rangle$ be given by 
the circled vertices. Then 
\[\tilde \alg_2 \cong 
 \bigoplus_{p\in \ZZ}\Hom_{\Cov(\alg_1)}\left(\bigoplus_{i=1}^n P_i\langle r_i\rangle,
                             \bigoplus_{i=1}^n P_i\langle r_i+p\rangle\right)
\]
is graded equivalent via the identity. There are three level partitions; the component 
of level $-1$ consists of the vertices $12,13,14,$ and $15$ along with the arrow 
$13\to 14$, $14\to 15$ and $15\to 12$. 

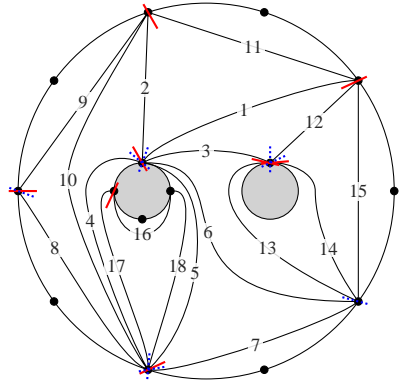
\begin{figure}
\centering
\begin{tikzpicture}[scale=.5]
\draw (0,0) circle[radius=5cm];
\filldraw[fill=gray!35]
      (-1.7,0) circle[radius=.75cm]
      (1.7,0) circle[radius=.75cm] ;
      
\foreach \a [count=\ai] in {0,36,...,324}
    \node[solid,name=o\ai] at (\a:5cm) {};

\node[solid,name=il] at ($(90:.75cm)+(-1.7,0)$) {};
\node[solid,name=il3] at ($(0:.75cm)+(-1.7,0)$) {};
\node[solid,name=il1] at ($(180:.75cm)+(-1.7,0)$) {};
\node[solid,name=il2] at ($(270:.75cm)+(-1.7,0)$) {};

\node[solid,name=ir] at ($(90:.75cm)+(1.7,0)$) {};

\draw 
  (il) .. controls +(55:1cm) and +(180:1cm) .. (o2) node[int] {1}
  (il) -- (o4) node[int] {2}
  (il) .. controls +(25:1cm) and +(155:1cm) .. (ir) node[int] {3}
  (il) .. controls +(160:1cm) and +(90:.5cm) .. ($(-1.7,0)+(180:1.5cm)$)
  ($(-1.7,0)+(180:1.5cm)$) .. controls +(270:1cm) and +(120:1cm) .. (o8) 
      node[int,pos=.2] {4}
  (il) .. controls +(350:.5cm) and +(120:.5cm) .. ($(0:1.2cm)+(-1.7,0)$)
  ($(0:1.2cm)+(-1.7,0)$) .. controls +(300:1cm) and +(45:2cm) .. (o8) 
      node[int] {5}
  (il) .. controls +(0:3cm) and +(180:6cm) .. (o10) node[int] {6}
  (o8) .. controls +(355:1cm) and +(215:1cm) .. (o10) node[int] {7}
  (o8) .. controls +(155:1cm) and ($(o6)!.25!(o8)$) .. (o6) node[int,pos=.7] {8}
  (o6) -- (o4) node[int] {9}
  (o8) .. controls +(127:1cm) and +(270:1cm).. ($(o6)!.4!(il)$)
      node[int,pos=1] {10}
  ($(o6)!.4!(il)$) .. controls +(90:1cm) and +(240:2cm) .. (o4)
  (o4) -- (o2) node[int] {11} 
  (ir) -- (o2) node[int] {12}
  (ir) .. controls +(180:1cm) and +(90:.25cm) .. ($(180:1.1cm)+(1.7,0)$)
  ($(180:1.1cm)+(1.7,0)$) .. controls +(270:1cm) and +(155:2cm) .. (o10)
      node[int] {13}
  (ir) .. controls +(0:.5cm) and +(90:.5cm) .. ($(0:1.2cm)+(1.7,0)$)
  ($(0:1.2cm)+(1.7,0)$) .. controls +(270:1cm) and +(125:1cm) .. (o10)
      node[int] {14}
  (o2) -- (o10) node[int] {15}
  (il1) .. controls +(270:1.5cm) and +(270:1.5cm) .. (il3) node[int] {16}
  (il1) .. controls +(180:1cm) and +(90:1cm) .. (o8) node[int] {17}
  (il3) .. controls +(0:1cm) and +(75:1cm) .. (o8) node[int] {18};
\draw[rcut] 
	($(o6)+(180:.25)$) -- ($(o6)+(0:.5)$)     
	($(il)+(300:.25)$) -- ($(il)+(120:.5)$)   
	($(o4)+(120:.25)$) -- ($(o4)+(300:.5)$)   
	($(o2)+(25:.25)$) -- ($(o2)+(205:.5)$)    
	($(o8)+(205:.25)$) -- ($(o8)+(25:.5)$)    
	($(ir)+(-14:.25)$) -- ($(ir)+(168:.5)$)   
	($(ir)+(185:.25)$) -- ($(ir)+(5:.5)$)    
	($(il1)+(65:.25)$) -- ($(il1)+(245:.5)$)  
	;
\draw[bcut,densely dotted] 
	($(il)-(140:.25)$) -- ($(il)+(140:.5)$)   
	($(o6)-(340:.25)$) -- ($(o6)+(340:.5)$)   
	($(o8)-(85:.25)$) -- ($(o8)+(85:.5)$)     
	($(il)+(250:.25)$) -- ($(il)+(70:.5)$)    
	($(ir)+(270:.25)$) -- ($(ir)+(90:.5)$)    
	($(o8)+(190:.25)$) -- ($(o8)+(10:.5)$)    
	($(o10)-(165:.25)$) -- ($(o10)+(165:.5)$)    
	($(ir)+(205:.25)$) -- ($(ir)+(25:.5)$)    
	;

\end{tikzpicture}
\caption{The solid red lines indicate the cut $\chi_1$, 
         the dashed blue lines $\chi_2$ from Example~\ref{ex:pants}.}
\label{fig:SurfaceEx1}
\end{figure}

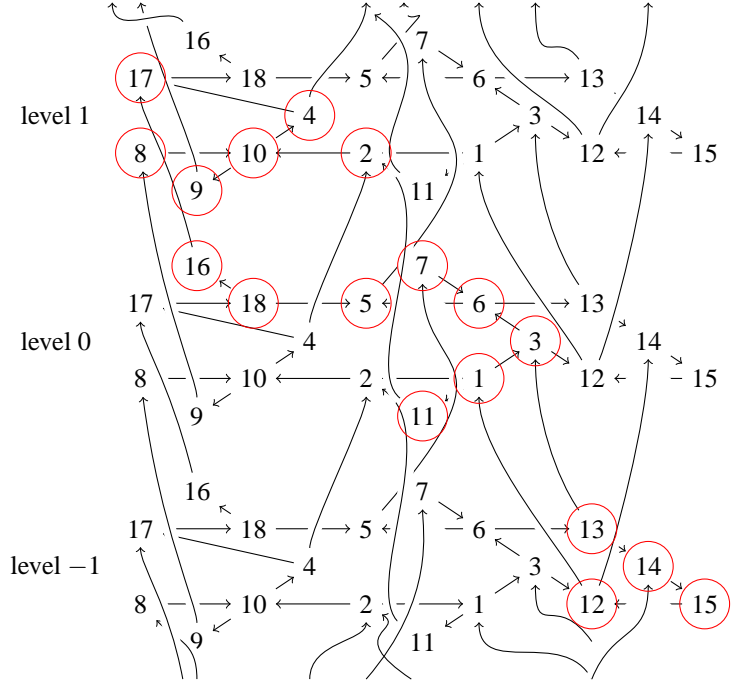
\begin{figure}
\centering

\begin{tikzpicture}[xscale=.75 ,yscale=.5]
	\node at (-2.5,6) {level 1};
	\node at (-2.5,0) {level 0};
	\node at (-2.5,-6) {level $-1$};

	\foreach \x/\y/\n in {4/2/7,0/2/16,-1/1/17,1/1/18,3/1/5,5/1/6,7/1/13,
	2/0/4,6/0/3,8/0/14,-1/-1/8,1/-1/10,3/-1/2,5/-1/1,7/-1/12,9/-1/15,0/-2/9,
	4/-2/11}
	{\node[name=p\n] at (\x,\y+6) {\n};
	 \node[name=0\n] at (\x,\y) {\n};
	 \node[name=n\n] at (\x,\y-6) {\n};
	}

	\foreach \b/\e in {
	8/10,10/9,
	2/10,10/4,
	2/1,1/11,
	1/3,3/12,
	4/17,18/5,
	6/5,7/6,
	3/6,6/13,
	13/14,
	14/15,15/12,
	17/18,18/16}
	{\path[->] (p\b) edge (p\e);
	 \path[->] (0\b) edge (0\e);
	 \path[->] (n\b) edge (n\e);
	}

	\foreach \b/\e in {n/0,0/p}{
		\draw[preaction={draw,white,line width=5}][->] 
			(\b9) .. controls +(90:1cm) and ($(\b17)+(.35,0)$) ..(\e8);
		\draw[preaction={draw,white,line width=5}][->] 
			(\b16) .. controls ($(\e8)+(.5,0)$) and +(270:1cm) ..(\e17);
		\draw[preaction={draw,white,line width=5}][->] 
			(\b4) .. controls +(90:1cm) and +(270:1cm) ..(\e2);
		\draw[preaction={draw,white,line width=5}][->] 
			(\b5) .. controls +(60:5.5cm) and +(270:3cm) ..(\e7);
		\draw[preaction={draw,white,line width=5}][->] 
			(\b11) .. controls +(130:1.75cm) and +(315:2cm) ..(\e2);
		\draw[preaction={draw,white,line width=5}][->] 
			(\b12) .. controls +(110:1.5cm) and +(270:2cm) .. (\e1);
		\draw[preaction={draw,white,line width=5}][->] 
			(\b12) .. controls +(75:1.5cm) and +(270:2cm) .. (\e14);
		\draw[preaction={draw,white,line width=5}][->] 
			(\b13) .. controls +(115:1.5cm) and +(270:2cm) .. (\e3);
	}
\draw[preaction={draw,white,line width=5}][->] 
	(p9) .. controls +(90:1cm) and ($(p17)+(.35,0)$) ..($(p8)+(0,4)$);
\draw[preaction={draw,white,line width=5}][->] 
	($(n9)-(0,1)$) .. controls +(90:1cm) and ($(n9)-(.35,0)$) ..(n8);
\draw[preaction={draw,white,line width=5}][->] 
	(p16) .. controls ($(p8)+(.5,4)$) and +(270:1cm) ..($(p17)+(-.5,2)$);
\draw[preaction={draw,white,line width=5}][->] 
	($(n16)-(.25,5)$) .. controls ($(n8)+(.5,0)$) and +(270:1cm) ..(n17);
\draw[preaction={draw,white,line width=5}][->] 
	(p4) .. controls +(90:1cm) and +(270:1cm) ..($(p2)+(0,4)$);
\draw[preaction={draw,white,line width=5}][->] 
	($(n4)-(0,3)$) .. controls +(90:1cm) and +(270:1cm) ..(n2);
\draw[preaction={draw,white,line width=5}][->] 
	(p5) .. controls +(60:3cm) and +(270:1cm) .. ($(p7)+(-.33,1)$);
\draw[preaction={draw,white,line width=5}][->] 
	($(n5)-(0,4)$) .. controls +(60:2.5cm) and +(270:1cm) ..(n7);
\draw[preaction={draw,white,line width=5}][->] 
	(p12) .. controls +(110:1.5cm) and +(270:2cm) .. ($(p1)+(0,4)$);
\draw[preaction={draw,white,line width=5}][->] 
	($(n12)-(0,2)$) .. controls +(110:1.5cm) and +(270:2cm) .. (n1);
\draw[preaction={draw,white,line width=5}][->] 
	(p12) .. controls +(75:1.5cm) and +(270:2cm) .. ($(p14)+(0,3)$);
\draw[preaction={draw,white,line width=5}][->] 
	($(n12)-(0,2)$) .. controls +(75:1.5cm) and +(270:2cm) .. (n14);
\draw[preaction={draw,white,line width=5}][->] 
	(p13) .. controls +(115:1.5cm) and +(270:2cm) .. ($(p3)+(0,3)$);
\draw[preaction={draw,white,line width=5}][->] 
	($(n13)-(0,3)$) .. controls +(115:1.5cm) and +(270:2cm) .. (n3);
\draw[preaction={draw,white,line width=5}][->] 
	(p11) .. controls +(130:1.75cm) and +(315:2cm) ..($(p2)+(.15,3.68)$);
\draw[->] 
	($(n11)-(.2,1)$) .. controls +(135:1.75cm) and +(315:1cm) ..(n2);

\foreach \n in {16,18,5,7,6,3,1,11}
	\node[sel] at (0\n) {};
\foreach \n in {17,4,8,9,10,2}
	\node[sel] at (p\n) {};
\foreach \n in {13,12,14,15}
	\node[sel] at (n\n) {};
\end{tikzpicture}
\caption{The quiver of $\Cov(\tilde \alg_1)$, the circled vertices are 
those such that $r(v)=i$ and determine a graded equivalence via the identity 
between $\tilde\alg_1$ and $\tilde\alg_2$ from Example~\ref{ex:pants}.}
\label{fig:cov1}
\end{figure}
\end{ex}

\begin{prop}\label{prop:boundedregions}
Let $\chi_1$ and $\chi_2$ be two admissible cuts of a surface $(S,M,T)$
such that $\tilde \alg_1$ and $\tilde \alg_2$ are graded equivalent via the
identity. The connected components of each level partition of
$\Cov(\tilde \alg_1)$ determine a connected region in $S$ bounded by the
sliding edges of $(\chi_1,\chi_2)$ and $\partial S$.
\end{prop}

\begin{proof} Let $Q^*$ be the quiver of $\Cov(\tilde \alg_1)$ and $\calc$ be a 
level connected component of $Q^*$. By definition, $\calc$ can touch other 
connected components only by bridge arrows, which are associated to sliding 
edges of $(\chi_1,\chi_2)$. Recall that arrows correspond to triangles in 
$(S,M,T)$. Because $\calc$ is a connected subgraph of $Q^*,$ $\calc$ must 
correspond to some contiguous collection of triangles in $(S,M,T)$, denote 
this collect by $\frakC$. Further, because we can only leave $\calc$ via 
bridge arrows, we must also have that $\frakC$ is bounded by $\partial S$ and 
sliding edges.
\end{proof}

Note that if $\calc$ consists of a single vertex, then $\frakC$ will consist of 
a single sliding edge, this edge must be two-sliding.  In all other cases 
$\frakC$ will have positive area. 

\begin{remark}\label{remark:boundedregions}
Proposition~\ref{prop:boundedregions} implies that any two edges $\te j$ and 
$\te k$ contained in the interior of the same bounded region must have the 
same value $r_j=r_k$, because these regions are determine by the connected 
components of that level.
\end{remark}

\begin{dfn}
The pair $(\chi_1,\chi_2)$ is called \df{equi-distributed} if for each boundary
component $\bcpt$, we have $|\im\cutt_1 \cap B|  = |\im\cutt_2\cap B|$, meaning the
the number of cuts in $\chi_1$ on $\bcpt$ is equal to the number of cuts in 
$\chi_2$ on $\bcpt$.
\end{dfn}

\begin{thm}\label{thm1}
Let $(S,M,T)$ be a triangulated bordered surface of genus 0 and $\alg_1$ and
$\alg_2$ surface algebras of type $(S,M,T)$.  Then $\tilde\alg_1$ and
$\tilde\alg_2$ are graded equivalent via the identity if and only if 
$(\chi_1,\chi_2)$ is equi-distributed.
\end{thm}

\begin{proof}
First we assume that $(\chi_1,\chi_2)$ is equi-distributed.
Set $Q_i$ to be the quiver of $\alg_i$, $Q$ the quiver of $(S,M,T)$, and
$Q^*$ the quiver associated to $\Cov(\tilde \alg_1)$. By determining the 
associated level partitions in $Q^*$ we will explicitly describe the function 
$r\colon Q_0 \to \ZZ$ so that we have
\[\alg_2 \xrightarrow{\ \cong \ }\bigoplus_{p\in \ZZ}
    \Hom_{\Cov(\alg_1)}\left(\bigoplus_{i=1}^n P_i\langle r_i\rangle,
                          \bigoplus_{i=1}^n P_i\langle r_i+p\rangle\right)
\]

Because of Proposition~\ref{prop:boundedregions} and
Remark~\ref{remark:boundedregions} it is sufficient to only determine the values
for $r$ near sliding edges. The value for other edges will be induced by the
choices at the sliding edges.

The process is to choose (at random) a bounded region and assign $r=0$ to each
internal arc of that region. Applying Remark~\ref{remark:configurations}, we then
proceed to assign values of $r$ to each sliding edge bounding the chosen region
as well as the neighboring regions. We then reiterate this process with each
neighboring region and so on. The primary work of the proof is to show that such
a choice is well defined for all of $S$.
Assume first that $S$ has at least two boundary components.

Let $\{\frakC_i\}$ $i=0,\dots,r$ be the bounded regions given by $(\chi_1,\chi_2)$
and let $r_i$ be the corresponding value of $r$ for $\frakC_i$. We now consider
$r$ as the function $r\colon S\to \ZZ$ by setting $r(x) = r_{i}$ for
$x\in\frakC_i$. Fix $i$ and a point $x_0\in \frakC_i$ let $\rho$ be a
non-contractible loop based at $x_0$, without loss of generality we may let $i=0$.
We may assume that $r(x_0)=0$. Because $S$ is genus zero, the loop divides $S$
into two parts, the inside (to the right) of $\rho$ and the outside (to the left)
of $\rho$. We want to show that as we travel along $\rho$, in either direction,
and apply Remark~\ref{remark:configurations} to determine the value of $r$ as we
change bounded regions, we recover that $r(x_0)=0$ as we cross back into
$\frakC_0$. Let $r_0'$ be value of $r$ as we cross back into $\frakC_0$.

For each sliding edge $\te{}$ intersecting $\rho$ we associate two integers
$\Delta_r \te{}$ and $\Delta_\chi \te{}$. Let $\frakC$ and $\frakD$ be the
components that are bound by $\te{}$ and $\frakD$ follows $\frakC$ with respect to
$\rho$, then $\Delta_r \te{}:= r(\frakD) - r(\frakC)$. Further, let $a_{\te{}}$ be
the number of local cuts from $\chi_1$ incident to $\te{}$ on boundary components
inside of $\rho$ and $b_{\te{}}$ the number of local cuts from $\chi_2$ on boundary
components inside of $\rho$ and incident to $\te{}$, we define 
$\Delta_\chi \te{} = b_{\te{}}-a_{\te{}}$. For each sliding
edge and choice of $\rho$, since $r$ is chosen as in
Remark~\ref{remark:configurations}, then $\Delta_\chi \te{} = -\Delta_r \te{}$.
This can be shown by considering cases. See Figures~\ref{fig:proof2sliding}
and~\ref{fig:proof1sliding}.

The number $\Delta_\chi\rho=\sum_i \Delta_\chi\te i$ measures the total change in
the number of cuts on the boundary components inside of $\rho$. Similarly,
$\Delta_r\rho = \sum_i \Delta_r \te i$ measures the total change in $r$ after one
iteration of $\rho$. Hence, if $(\chi_1,\chi_2)$ is equi-distributed, then
$\Delta_\chi\rho = 0$. Therefore, $\Delta_r\rho=-\Delta_\chi\rho=0$.  It follows that 
$r_0=0$, as desired. Because $\rho$ is arbitrary, we see that the choice of 
$r$ given by Remark~\ref{remark:configurations} is well-defined.

Conversely, assume that $(\chi_1,\chi_2)$ is not equi-distributed.  Then
in the above analysis we must have $\Delta_r\rho=-\Delta_\chi\rho\neq 0$ for some 
loop $\rho$.  It follows that there is no consistent way to define the function 
$r\colon Q_0\to \ZZ$. It follows that $\tilde\alg_1$ and $\tilde\alg_2$ are not
graded equivalent.
\end{proof}

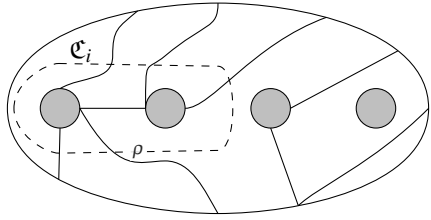
\begin{figure}
\centering
\begin{tikzpicture}[scale=.35]
\draw (8,0)  .. controls +(90:3) and +(0:3) .. (0,4)
             node[name=a,pos=0,coordinate] {a} 
             node[name=b,pos=.3,coordinate] {b}
             node[name=c,pos=.6,coordinate] {c} 
             node[name=d,pos=1,coordinate] {d}
      (0,4)  .. controls +(180:3) and +(90:3) .. (-8,0)
             node[name=e,pos=.3,coordinate] {e}
             node[name=f,pos=.6,coordinate] {f} 
             node[name=g,pos=1,coordinate] {g}
      (-8,0) .. controls +(270:3) and +(180:3) .. (0,-4)
             node[name=h,pos=.4,coordinate] {h}
             node[name=i,pos=.6,coordinate] {i}
             node[name=j,pos=1,coordinate] {j}
      (0,-4) .. controls +(0:3) and +(270:3) .. (8,0)
             node[name=k,pos=.3,coordinate] {k}
             node[name=l,pos=.6,coordinate] {l};
\filldraw[fill=gray!50] (-6,0) circle[radius=.75] ;
\filldraw[fill=gray!50] (-2,0) circle[radius=.75] ;
\filldraw[fill=gray!50] (2,0) circle[radius=.75] ;
\filldraw[fill=gray!50] (6,0) circle[radius=.75] ;
\draw (-6,.75) .. controls +(30:1) and +(270:1) .. (-4,2)
	    (-4,2) .. controls +(90:1) and +(220:1) .. (e)
	    (-6,-.75)  -- (h)
	    (-6+.75,0) .. controls +(300:3) and +(180:1) .. (-2,-2)
	    (-2,-2) .. controls +(0:1) and +(120:1) .. (j)
	    (-6+.75,0) -- (-2-.75,0)
	    (-2-.75,0) .. controls +(90:2) and +(220:1) .. (-2,2)
	    (-2,2) .. controls +(40:2) and +(270:1) .. (d)
	    (-2+.75,0) .. controls +(0:1) and +(200:3) .. (c)
	    (k)  -- (2,-.75)
	    (b) -- (2+.75,0)
	    (a) .. controls +(225:3) and +(45:1) .. (k);
\node[below] at ($(f)+(.85,.25)$) {$\frakC_i$};
\draw[dashed,xscale=.5] 
     ($(g)+(.5,0)$) .. controls +(90:1) and +(180:1) .. ($(f)-(0,1)$)
     ($(f)-(0,1)$) .. controls +(0:1) and +(180:1) .. ($(j)!.7!(d)$)
     ($(j)!.7!(d)$) .. controls +(0:1.5) and +(0:1.5) .. ($(j)!.3!(d)$)
     ($(j)!.3!(d)$) .. controls +(180:1) and +(0:1) .. ($(h)+(0,1)$) 
          node[int] {$\rho$}
     ($(h)+(0,1)$) .. controls +(180:1) and +(270:1) .. ($(g)+(.5,0)$);
\end{tikzpicture}
\caption{The surface $S$ partitioned by $\{\frakC_i\}$ and the loop $\rho$}
\label{fig:proofLoop}
\end{figure}

\begin{remark}
We remark that the above theorem does not hold for higher genus. Let $S$ be the
torus with one boundary component. Let $M$ be a single point on the boundary and
consider the triangulation $T$ in Figure~\ref{torusTriangulation}. Because there
is only one boundary component, Propotion~\ref{prop:graded equiv} would imply that
any two admissible cuts should be graded equivalent via the identity. However, the
cuts $\chi_{1,2}\chi_{3,4}$ and $\chi_{1,2}\chi_{4,1}$ are easily shown to not
be graded equivalent via the identity. Let $\alg_1$ be given by
$\chi_{1,2}\chi_{3,4}$ and $\alg_2$ be given by $\chi_{1,2}\chi_{4,1}$. Because
the induced weight on the arrows $1\to2$, $2\to3$ and $1\to3$ does not change
between the two cuts, we must have $r_1=r_2=r_3$, where $r_i$ is as in
Propostion~\ref{prop:graded equiv}. Additionally, because the weight on the arrow
$4\to1$ changes we must have $r_1\neq r_4$, but the weight on $2\to4$ does not
change so $r_2=r_4$, hence we must also have $r_1=r_4$, a contradiction.

\begin{figure}[h]
\centering
\begin{tikzpicture}
\node[coordinate,name=t1] at (1,1) {};
\node[coordinate,name=t2] at (-1,1) {};
\node[solid,name=b1] at (1,-1) {};
\node[coordinate,name=b2] at (-1,-1) {};
\draw (t1) -- node[int] {1} 
      (t2) -- node[int] {2} 
      (b2) -- node[int] {1} 
      (b1.center) -- node[int] {2} (t1)
      (b1) .. controls +(170:2) and +(225:1) ..(t1) node[int] {3}
      (t1) .. controls +(200:2) and +(45:1) .. (b2) node[int] {4};
\filldraw[fill=gray!40](b1) .. controls +(150:1.5) and +(120:1.5) .. (b1);
\end{tikzpicture}
\caption{A triangulation of the torus with one boundary component.}
\label{torusTriangulation}
\end{figure}
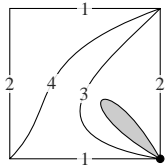
\end{remark}

\section{Boundary Permutations}
At this point we have determined that we get graded equivalent algebras when we
permute local cuts along a fixed boundary component. In this section we will
show that we can also permute cuts among different boundary components. 

We define the mapping class group of $(S,M)$ as in \cite{ASS}. Set
$\Homeo^+(S,M)$ to be the set of orientation preserving homeomorphism from $S$
to $S$ that send $M$ to $M$. Note that if a boundary component $C_1$ is mapped
to a component $C_2$, then the two components must have the same number of
marked points. We say that a homeomorphism $f$ is isotopic to the identity
relative to $M$, if $f$ is isotopic to the identity via a homotopy that fixes
$M$ point-wise. Then we set $\Homeo_0(S,M)$ to be the homeomorphisms isotopic to
the identity relative to $M$. The \df{mapping class group} of $(S,M)$ is
\[ \MG(S,M) = \Homeo^+(S,M)/\Homeo_0(S,M) \]

For $f\in \MG(S,M)$ we define $f$ at an admissible cut
$(S,M\dag,T\dag)$ by setting $f(\chi_{i,j}) = \chi_{f(i),f(j)}$ for each local
cut. By construction this induces a graded isomorphism of $\tilde \alg_{T\dag}$
and $\tilde \alg_{f(T\dag)}$ because it explicitly sends arrows of weight one to
arrows of weight one.

\begin{thm}\label{thm2} Let $(S,M,T)$ be a triangulated bordered surface of genus
0 and $\alg_1$ and $\alg_2$ be surface algebras of type $(S,M,T)$
coming from admissible cuts $\cutt_1$ and $\cutt_2$. Then $\tilde\alg_1$ and $\tilde\alg_2$
are graded equivalent if there is an element $f\in \MG(S,M)$ such that $f$ induces
a quiver automorphism on $Q_T$ and $(\cutt_1,f(\cutt_2))$ or $(f(\cutt_1),\cutt_2)$ 
are equi-distributed.
\end{thm}
\begin{proof} 
Assume $(\chi_1,f(\chi_2))$ is equi-distributed,
hence $\tilde\alg_1$ and $f(\tilde\alg_2)$ are graded equivalent by 
Theorem~\ref{thm1}. By construction the extension of $f$ to the cut surface 
induces a graded isomorphism of $\tilde\alg_2$ and $f(\tilde \alg_2)$. 
It follows that $\tilde\alg_1$ and $\tilde \alg_2$ are graded equivalent.
\end{proof}

\begin{remark}
Note that the the theorem excludes the use of the Dehn twists in $\MG(S,M)$ In particular,
this is because the Dehn twists can never change the configuration of the cuts in $(S,M,T)$.  
In general, a mapping class $f\in \MG(S,M)$ satisfying Theorem~\ref{thm2} will have to leave
$T$ invariant,  the set of all such $f$ will be a small subset of $\MG(S,M)$.
\end{remark}

On the other hand, let $\tilde\alg_1$ and $\tilde\alg_2$ be graded equivalent algebras 
coming from admissible cuts $\chi_i$  of a surface $(S,M,T)$ such that the map 
$f\colon \tilde\alg_1\to\tilde\alg_2$ induces an isomorphism of quivers 
$f\colon \tilde Q_1\to \tilde Q_2$.
Because arrows are associated to triangles, $f$ induces a map $f_S\colon(S,M,T)\to (S,M,T)$. 
To understand what this map is, we consider $(S,M,T)$ as a CW-complex where the
0-skeleton is $M$, the 1-skeleton is given by $T$ and the boundary segments, and
the 2-skeleton is given by the ideal triangles. For convenience we use the following 
definition.
\begin{dfn}
Let $(S,M,T)$ be a triangulated surface without punctures. There are three
triangle types, we call those triangles with two edges in the boundary \df{corner
triangles}, triangles with one edge in the boundary \df{basic triangles}, and
triangles with no edge in the boundary \df{internal triangles}. Notice that there
is a unique edge in $T$ associated to each corner triangle.
\end{dfn}
Before analyzing the map $f_S$, we note the following fact. Given
a two dimensional finite CW-complex $S$ and $f_1$ a continuous self mapping on the
one-skeleton of $S$, there is a continuous map $f\colon S\to S$ that restricts to
$f_1$. This map is given by considering barycentric coordinates on the
homeomorphic image of each face into a convex open subset of $\RR^2$. Hence, it is
enough to understand $f_S$ on $T$ and boundary segments. 

The induced map on $(S,M,T)$ is given by first fixing and defining the map on a 
representative for each isotopy class in $T$.
Notice that $Q_T=\tilde Q_1=\tilde Q_2$ by assumption. Hence, we can view $f$ as an
automorphism on $Q_T= (Q_0, Q_1)$. Considering $f$ as a map on vertices
$Q_0$, we define $f_1$ on the edges of the triangulation by $f_1(\te i) = \te
j$ when $f(i) = j$ in $Q_0$. Note that $f$ preserves arrow orientation in $Q$
because it is a quiver automorphism. So, in the surface, if $\te i$ is incident to
$\te j$ with $\te j$ following $\te i$ in the \orientationofthearrow direction,
then $f_1(\te j)$ is incident to and follows $f_1(\te i)$ in the
\orientationofthearrow direction. It follows that $f_1$ preserves triangle types
and orientation, that is, the edges defining a basic, internal or corner triangle
will be sent to edges defining a basic, internal, or corner triangle respectively
and further those edges will be the same relative orientation. We can extend the
definition of $f_1$ to boundary segments, because of this preservation of triangle
type, as follows. Let $\triangle$ be a basic triangle with edges $\te i$, $\te j$ and boundary
segment $b$, we define $f_1(b)$ to be the boundary segment incident to $f_1(\te
i)$ and $f_1(\te j)$. Similarly, $f_1$ maps the corner triangle with edge $\te i$
and boundary segments $b$ and $b'$ with $b$ following $b'$ in the
\orientationofthearrow direction to the corner triangle with edge $f_1(\te i)$ and
boundary segments $f_1(b)$ and $f_1(b')$ with $f_1(b)$ following $f_1(b')$ in the
\orientationofthearrow direction. By construction this map will preserve the
on the 1-skeleton, hence the induced map $f_S$ will preserve the
orientation of $S$. As a result we have the following partial converse to Theorem~\ref{thm2}.

\begin{thm}\label{thm3}
Let $\tilde\alg_1$ and $\tilde\alg_2$ be graded equivalent algebras 
coming from admissible cuts $\chi_1$ and $\chi_2$  of a surface $(S,M,T)$ 
such that the map $f\colon \tilde\alg_1\to\tilde\alg_2$ induces an 
automorphism on $Q_T$.  Then, there is mapping class $f_S\in \MG(S,M)$ that 
induces the graded isomorphism $f$ and $(f(\chi_1),\chi_2)$ are equi-distributed.
\end{thm}

\begin{proof}
Let $f_S$ be given as in the above discussion and let $\tilde\alg'$ be the
graded algebra given by $f_S(\chi_1)$. Note that $\tilde\alg'$ need not be
$\tilde\alg_2$, but, by construction it will be graded equivalent to
$\tilde\alg_2$, call the corresponding equivalence $g$. Indeed, $\tilde\alg_2$
and $\tilde\alg'$ are graded equivalent via the identity. That is, $g$ induces
the identity map on the level of quivers. This follows immediately by carefully
unwinding the definitions. By assumption $\tilde\alg_1$, $\tilde\alg_2$, and
$\tilde\alg'$ have quiver $Q_T$ . Let $\phi\colon Q_T\to Q_T$ denote the
automorphism induced by $f$. By construction, $f_S$ induces the same a quiver
automorphism $\phi$. Hence, we have the following commutative diagrams

\[\begin{tikzpicture}
{[]
\node[name=1] at (0,0) {$\tilde\alg_1$};
\node[name=2] at (2,0) {$\tilde\alg_2$};
\node[name=p] at (1,-1) {$\tilde\alg'$};
\path[->] (1) edge node[int]{$f$} (2) edge node[int] {$f_S$} (p) (p) edge node[int] {$g$} (2);
}
{[xshift=4cm]
\node[name=1] at (0,0) {$Q_T$};
\node[name=2] at (2,0) {$Q_T$};
\node[name=p] at (1,-1) {$Q_T$};
\path[->] (1) edge node[int]{$\phi$} (2) edge node[int] {$\phi$} (p) (p) edge (2);
}
\draw[-implies,double,double equal sign distance] (2.5,-.5) -- (3.5,-.5);
\end{tikzpicture}\]
Therefore, the map induced by $g$ must be the identity map. It follows from Theorem~\ref{thm1}, that
$(f(\chi_1),\chi_2)$ are equi-distributed. 
\end{proof}

\begin{ex}\label{ex mg}
We give an example of a graded equivalence given by a non-trivial mapping class.
Let $(S,M,T)$ be given as in Figure~\ref{fig excover}. The quiver of the 
triangulation is $Q_T$

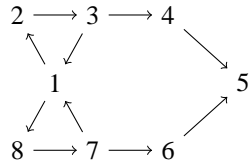
\[\begin{tikzpicture}[yscale=.75]
\node[name=2] at (1,1.2) {2};
\node[name=3] at (2,1.2) {3};
\node[name=4] at (3,1.2) {4};
\node[name=5] at (4,0) {5};
\node[name=6] at (3,-1.2) {6};
\node[name=7] at (2,-1.2) {7};
\node[name=8] at (1,-1.2) {8};
\node[name=1] at (1.5,0) {1};
\foreach \x/\y in {
  3/1,2/3,1/2,3/4,4/5,
  6/5,7/6,8/7,7/1,1/8}
 \draw[->] (\x) -- (\y);
\end{tikzpicture}.\]
Let $\chi_1 = \chi_{1,2}\chi_{7,1}$ and $\chi_2 = \chi_{8,7}\chi_{3,1}$, given by
the red and blue lines respectively. The corresponding surface algebras $\alg_1$
and $\alg_2$ are derived equivalent by Theorem~\ref{mainthm}. The required
automorphism of the surface $f$ can be realized by a rotation of the universal
cover of $S$ that fixes a lift of the $\te1$, see Figure~\ref{fig excover}. This
map will induce the quiver automorphism given by the map on vertices
\begin{align*}
1 &\mapsto 1 & 2&\mapsto 8 & 3&\mapsto7 & 4&\mapsto6\\
5 &\mapsto 5 & 6&\mapsto 2 & 7&\mapsto3 & 8&\mapsto2.
\end{align*}
Note that the image of $\chi_1$ under this map is not $\chi_2$, but 
$(f(\chi_1),\chi_2)$ is equi-distributed.

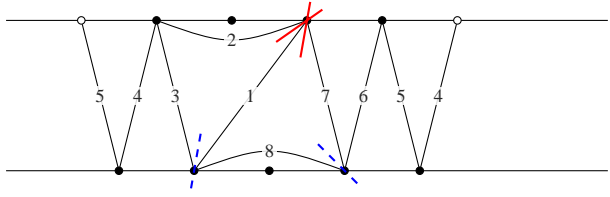
\begin{figure}
\centering

\begin{tikzpicture}
\draw (0,0) -- (8,0) (0,-2) -- (8,-2);
\foreach \x in {2,3,4,5}
  \node[name=o\x,solid] at (\x,0) {};
\foreach \x in {1,2,3,4,5}
  \node[name=i\x,solid] at (\x+.5,-2) {};
\node[name=o1,pair] at (1,0) {};
\node[name=o6,pair] at (6,0) {};
\draw 
  (o2) ..controls +(-20:1cm) and +(200:1cm) .. (o4)
    node[int] {2}
  (o2) -- (i2)
    node[int] {3}
  (o4) -- (i2)
    node[int] {1}
  (i2) ..controls +(20:1cm) and +(160:1cm) .. (i4)
    node[int] {8}
  (i4) -- (o4)
    node[int] {7}
  (i4) -- (o5)
    node[int] {6}
  (o2) -- (i1) 
    node[int] {4}
  (o5) -- (i5)
    node[int] {5}
  (i1) -- (o1)
    node[int] {5}
  (i5) -- (o6)
    node[int] {4};
\draw[bcut] 
  ($(i2)+(80:.5cm)$) -- ($(i2)-(80:.25cm)$)
  ($(i4)+(135:.5cm)$) -- ($(i4)-(135:.25cm)$);
\draw[rcut]
  ($(o4)+(215:.5cm)$) -- ($(o4)-(215:.25cm)$)
  ($(o4)+(260:.5cm)$) -- ($(o4)-(260:.25cm)$);
\end{tikzpicture}
\caption{The universal cover of the annulus. The solid red lines represent $\chi_1$, 
dashed blue lines represent $\chi_2$ from example \ref{ex mg}}\label{fig excover}
\end{figure}
\end{ex}

Theorem~\ref{thm2} does not tell us how to identify the homeomorphism of the
surface that gives rise to the graded equivalence. Naturally, we want to determine
which automorphisms of the surface determine a graded equivalence. A minimal
combinatorial description can be given if we ignore some of the surface structure
and consider the automorphism in combinatorial terms of the marked points,
boundary components and triangles. In these terms, finding automorphisms that
induce a graded equivalence is equivalent to finding permutations of the boundary
components and of the marked points such that the corresponding map on the set
of triangles sends neighboring triangles to neighboring triangles and boundary
components to boundary components. Under the permutation of boundary components, 
a component can only be sent to another component with the same local configuration
of triangles incident to the component.   Similarly, a marked point must be sent to 
a marked point with the same number and type of incident triangles, these triangles
must occur in the same order in the \orientationofthearrow direction.

Recall that we may associate a cluster algebra to a triangulated surface, see
\cite{FST}. The mapping classes of $(S,M,T)$ that correspond to graded
equivalences will correspond to cluster automorphisms, defined in \cite{ASS},
which fix (up to a permutation) the cluster corresponding to the triangulation.

\section{Derived Equivalence}
All of this work to describe the graded equivalences of surface algebras has 
been done with the goal of determining derived equivalences.  We restate a 
theorem of \cite{AO} in terms of surface algebras.

\begin{thm}[{\cite[Theorem 5.6]{AO}}]\label{thm AO}
Let $\alg_1$ and $\alg_2$ be surface algebras coming from admissible cuts
$\chi_1$ and $\chi_2$. Then $\alg_1$ and $\alg_2$ are derived equivalent if and
only if $\tilde \alg_1$ and $\tilde \alg_2$ are graded equivalent.
\end{thm}

We can now reformulate the theorems of section 2 and 3.

\begin{thm}\label{mainthm}
Let $\alg_1$ and $\alg_2$ be surface algebras of type $(S,M,T)$ coming from
admissible cuts $\chi_1$ and $\chi_2$ respectively. Then $\alg_1$ and $\alg_2$
are derived equivalent if there is an element $f\in \MG(S,M)$ such
that $f$ induces a quiver automorphism on $Q_T$ and $(\chi_1,f(\chi_2))$ or 
$(f(\chi_1),\chi_2)$ are equi-distributed.
\end{thm}

The proof of Theorem~\ref{thm AO} explicitly describes the tilting object 
associated to a given graded equivalence.  If we have
\[\tilde \alg_2 \xrightarrow{\ \sim \ }\bigoplus_{p\in \ZZ}
    \Hom_{\Cov(\tilde \alg_1)}\left(\bigoplus_{i=1}^n P_i\langle r_i\rangle,
                          \bigoplus_{i=1}^n P_i\langle r_i+p\rangle\right)
\]
Then $ \bigoplus_{i=1}^n F^{-r_i}P_i$ is the tilting object in $\cald^b(\cmod
\alg_1)$ that gives the derived equivalence of $\alg_1$ and $\alg_2$. Where
$F:= \mathbb{S}[-2]$ with $\mathbb{S}$ the Serre functor of $\cald^b(\cmod
\alg_1)$.

\section{Reflections of gentle algebras}
In the theory of cluster algebras, quiver mutation plays an important role.
For cluster-algebras from surfaces this mutation can be realized in the surface
as a flip of an edge in the triangulation.  In this section we will show that a
similar idea exists for surface algebras via the reflections in quivers of 
gentle algebras.  These reflections induced derived equivalences via an explicit
tilting object in module category. In contrast to the derived equivalences obtained 
via mapping classes, the derived equivalences obtained via reflections 
need not be between surface algebras of the same triangulation, in fact most are not.
Additionally, reflections allow us to describe some derived equivalences of 
surface algebras in terms of tilting modules, instead of tilting objects in the
derived category.  

\subsection{Definitions} We begin by recalling definitions.

\begin{dfn}
The mutation of $Q$ at vertex $j$, denoted $\mu_j(Q)$, is the quiver obtained
from $Q$ by the following procedure:
\begin{enum}
\item Reverse each arrow incident to $j$.
\item For all paths $i\to j\to k$ in $Q$, we introduce an arrow $i\to k$ in
      $\mu_j(Q)$. 
\item Delete all 2-cycles that may have been generated. 
\end{enum}
\end{dfn}

In a triangulated surface without punctures $(S,M,T)$, each edge $\te{}$ of the
triangulation is contained in exactly two distinct triangles that form a
quadrilateral in which $\te{}$ is a diagonal. The mutation of the triangulation,
$\mu_{\te{}}(T)$, is given by $T\setminus\{\te{}\}\cup\{\te{}'\}$ where $\te{}'$ is the other
diagonal of the quadrilateral containing $\te{}$. If $j\in Q_0$ corresponds to $\te j$,
then $\mu_j(Q)$ is the quiver of $\mu_{\te{j}}(T)$.

\begin{dfn}
A finite dimensional $k$-algebra $\alg$ is called gentle if the bound quiver
$(Q,I)$ associated to $\alg$ satisfies:
\begin{enumerate}
\item For each $i\in Q_0$, $\#\{\alpha\in Q_1:s(\alpha) = i\}\leq 2$ and
      $\#\{\alpha\in Q_1:t(\alpha) = i\}\leq 2$.
\item For each $\beta\in Q_1$, $\#\{\alpha\in Q_1:s(\beta)=t(\alpha)$ and
      $\alpha\beta\not\in I\}\leq1$ and $\#\{\gamma\in Q_1:s(\gamma)=t(\beta)$
      and $\beta\gamma\not\in I\}\leq 1$
\item The ideal $I$ is generated by paths of length 2.
\item for each $\beta\in Q_1$, $\#\{\alpha\in Q_1:s(\beta)=t(\alpha)$ and
      $\alpha\beta\in I\}\leq1$ and $\#\{\gamma\in Q_1:s(\gamma)=t(\beta)$ and
      $\beta\gamma\in I\}\leq 1$
\end{enumerate}
\end{dfn}
Surface algebras are gentle \cite{AG,DS}. For the remainder of the section we 
assume that $Q$ is a gentle quiver with relations $I$.
\begin{dfn}
Let $i$ be a vertex of $Q$ such that for each arrow $\alpha\in Q_1$ with
$s(\alpha)=i$ there exists $\beta_\alpha\in Q_1$ with $t(\beta_\alpha)=i$ and
$\beta_\alpha\alpha\not\in I$. The reflection of $Q$ at vertex $i$, denoted
$R_i(Q):=(Q_0',Q_1',s',t')$, is the quiver with relations $I'$ obtained from $Q$
as follows:
\begin{itemize}
\item The vertices and arrows of $Q'$ are the vertices and arrows of $Q$, that
      is $Q_0' =Q_0$ and $Q_1'=Q_1$, only the maps $s$ and $t$ change.
\item We define
	\begin{align*}
	s'\alpha & := 
		\begin{cases}
		  i &\IF t(\alpha) = i,\\
		  s(\beta_\alpha)& \IF s(\alpha) = i,\\
		  s(\alpha) & \text{ otherwise},
		\end{cases}\\
	t'\alpha & :=
		\begin{cases}
		s(\alpha) & \IF t(\alpha) = i,\\
		i &\IF \exists \beta\in Q_1 \text{ such that } t(\beta) = i \text{ and } 
		    s(\beta) = t(\alpha) \text{ and }\alpha\beta \in I,\\
		t\alpha & \text{ otherwise}.
		\end{cases}
	\end{align*}
\item We define $I':= I_1\cup I_2 \cup I_3$ where
 \begin{align*}
  I_1  &= \{\beta\alpha: t(\alpha) = i \text{ and } \exists \gamma\in Q_1 
    \text{ such that } \gamma\neq\alpha,
	t(\gamma) = i \text, s(\gamma) =t(\beta), \text{ and }\beta\gamma \in I\},\\
  I_2 &= \{\alpha\beta\in I:t(\beta)\neq i \text{ and } s(\beta)\neq i\},\\
  I_3 &= \{ \beta_\alpha\alpha : s(\alpha) = i\}.
  \end{align*}
\end{itemize}
Notice that the arrow $\alpha\in Q$ is also denoted $\alpha\in R_i(Q)$, the only 
difference is the definition of the source and target function. When we define the 
relations in $R_i(Q)$, we use the composition of arrows in $R_i(Q)$ but
use the original functions $s$ and $t$ from $Q$ when selecting which arrows are in a relation. 
Many examples will be given below. 
\end{dfn}

\begin{dfn}
Dually, we define the co-reflection at $i$. Let $i$ be a vertex of $Q$ such that
for each arrow $\alpha\in Q_1$ with $t(\alpha)=i$ there exists $\beta_\alpha\in
Q_1$ with $s(\beta_\alpha)=i$ and $\alpha\beta_\alpha\not\in I$. The
coreflection of $Q$ at vertex $i$, denoted $R_i^-(Q):=(Q_0',Q_1',s',t')$, is the
quiver with relations $I'$ obtained from $Q$ as follows:
\begin{itemize}
\item The vertices and arrows of $Q'$ are the vertices and arrows of $Q$, that
      is $Q_0' =Q_0$ and $Q_1'=Q_1$, only the maps $s$ and $t$ change.
\item We define
	\begin{align*}
	s'\alpha & := 
		\begin{cases}
		  t(\alpha) &\IF s(\alpha) = i,\\
		  i &\IF \exists \beta \in Q_1 \text{ such that } s(\beta)=i
		    \text{ and } \beta\alpha\in I,\\
		  s(\alpha) & \text{ otherwise}.
		\end{cases}\\
	t'\alpha & :=
		\begin{cases}
		i & \IF s(\alpha) = i,\\
		t(\beta_\alpha) &\IF t(\alpha)=i,\\
		t(\alpha) & \text{ otherwise}.
		\end{cases}
	\end{align*}
\item We define $I':= I_1\cup I_2 \cup I_3$ where
 \begin{align*}
  I_1  &= \{\alpha\beta: s(\alpha) = i \text{ and } \exists \gamma\in Q_1 
     \text{ such that } \gamma\neq\alpha,
	 s(\gamma) = i \text, t(\gamma) =s(\beta), 
	 \text{ and }\gamma\beta \in I\},\\
  I_2 &= \{\alpha\beta\in I:s(\beta)\neq i \text{ and } t(\beta)\neq i\},\\
  I_3 &= \{ \alpha\beta_\alpha : t(\alpha) = i\}.
  \end{align*}
\end{itemize}
\end{dfn}

The reflection of a quiver gives a Brenner-Butler tilt of the corresponding
algebra. Let $(Q,I)$ be the quiver of a gentle algebra $\alg$ and $(Q',I') =
R_i(Q)$, then $kQ'/I' \cong \End(T)$ where
\[ T = \tau^{-1} S_i \oplus \bigoplus_{\substack{j\in Q_0\\ j\neq i}}\alg e_j,\]
with $S_i$ the simple representation at $i$.

\subsection{Mutations and Reflections} We will show that most reflections can be
described in terms of mutations and admissible cuts.

\begin{thm}\label{prop dict}
Let $Q$ be a quiver of a surface algebra given by an admissible cut of an
algebra from a triangulated surface with quiver $\tilde Q$. 
If $i$ is not the source in $Q$ of a relation and $R_i$ is defined, then there is an 
admissible cut of $\mu_i(\tilde Q)$ 
that gives $R_i(Q)$. Dually, if $i$ is not the target of a relation and $R_i^-$
is defined, then there is an admissible cut of $\mu_i(\tilde Q)$ that gives 
$R_i^-(Q)$.
\end{thm}

\begin{remark}
The definition of $\mu_i$ and $R_i$ are local to the vertex $i$. Specifically,
the construction of $Q$ from the triangulation of a surface is sufficiently
restrictive that the only possible changes between $Q$ and either $\mu_i(Q)$ or
$R_i(Q)$ can occur in arrows that start or end within a two vertex neighborhood
of $i$. Hence, in the proof of the proposition it suffices to only consider the 
local configurations of $Q$ near $i$. 
\end{remark}

\begin{proof}
We only present those configurations without double arrows.  In each
configuration, we can retrieve those configurations with double edges by 
identifying the white vertices.  In the very first configuration we may also 
identify the black vertices, but we may not identify the white and black vertices
at the same time. 

 Because we only consider surface algebras of admissible cuts,
there are no overlapping relations in $Q$. This follows from that the fact that
there are no overlapping relations in $\tilde Q$ outside of the 3-cycles which
are cut in $Q$. Hence, there are 10 possible local configurations near $i$ at
which we can reflect and satisfy the assumptions of the theorem. We will provide
a dictionary for these 10 configurations. First note that if $i$ is a sink that
is not the end of any relations, then mutation and reflection have the exact same
effect on $Q$. We will not include this case below. Throughout the proof,
relations will be indicated by dashed lines.

First, assume that $i$ is the source of at least one arrow and is not the target
of any relation. Because of the restrictions on where we may reflect, we get the
following three possibilities.
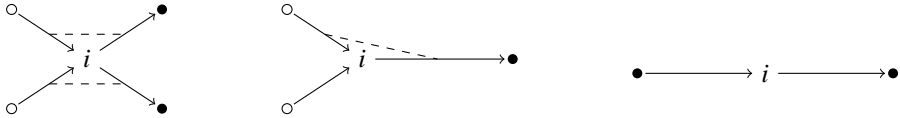
\[\begin{tikzpicture}[baseline=0,yscale=.66]
\node[wert,name=A] at (0,0) {};
\node[vert,name=B] at (2,0) {};
\node[wert,name=C] at (0,-2) {};
\node[vert,name=D] at (2,-2) {};
\node[name=E] at (1,-1) {$i$};
\path[->] (A) edge (E)   (C) edge (E)    (E) edge (B) edge (D);
\path[dashed] ($(A)!.5!(E)$) edge ($(E)!.5!(B)$) 
    ($(C)!.5!(E)$) edge ($(E)!.5!(D)$);
\end{tikzpicture}\hspace{.6in}
\begin{tikzpicture}[baseline=0,yscale=.66]
\node[wert,name=A] at (0,0) {};
\node[name=B] at (1,-1) {$i$};
\node[vert,name=C] at (3,-1) {};
\node[wert,name=D] at (0,-2) {};
\path[->] (A) edge (B)   (D) edge (B)    (B) edge (C);
\path[dashed] ($(A)!.5!(B)$) edge ($(B)!.5!(C)$);
\end{tikzpicture}\hspace{.6in}
\begin{tikzpicture}[baseline=0,scale=.85]
\node[vert,name=A] at (0,-1) {};
\node[name=B] at (2,-1) {$i$};
\node[vert,name=C] at (4,-1) {};
\path[->] (A)  edge (B) (B) edge (C);
\end{tikzpicture}\]
The corresponding reflections are 
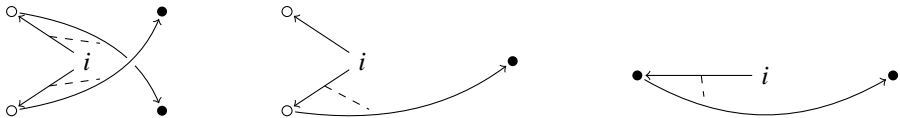
\[\begin{tikzpicture}[baseline=0,yscale=.66]
\node[wert,name=A] at (0,0) {};
\node[vert,name=B] at (2,0) {};
\node[wert,name=C] at (0,-2) {};
\node[vert,name=D] at (2,-2) {};
\node[name=E] at (1,-1) {$i$};
\path[->]  
   (A) edge[bend left] node[name=R1,pos=.5] {} (D)  
   (C) edge[preaction=outline,bend right]  node[name=R2,pos=.5] {} (B)  
   (E) edge (A) edge (C);
\path[dashed] ($(E)!.5!(A)$) edge (R1) ($(E)!.5!(C)$) edge (R2);
\end{tikzpicture}\hspace{.6in}
\begin{tikzpicture}[baseline=0,yscale=.66]
\node[wert,name=A] at (0,0) {};
\node[name=B] at (1,-1) {$i$};
\node[vert,name=C] at (3,-1) {};
\node[wert,name=D] at (0,-2) {};
\path[->] (B) edge (A) edge (D) (D) edge[bend right] node[pos=.35,name=R]{} (C);
\path[dashed] ($(B)!.5!(D)$) edge (R);
\end{tikzpicture}\hspace{.6in}
\begin{tikzpicture}[baseline=0,scale=.85]
\node[vert,name=A] at (0,-1) {};
\node[name=B] at (2,-1) {$i$};
\node[vert,name=C] at (4,-1) {};
\path[->] (B)  edge (A) (A) edge[bend right] node[pos=.25,name=R] {} (C);
\path[dashed] ($(B)!.5!(A)$) edge (R);
\end{tikzpicture}\] 
The quivers $\tilde Q$ corresponding to $Q$ before the reflection are
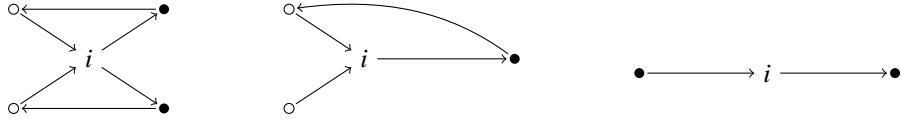
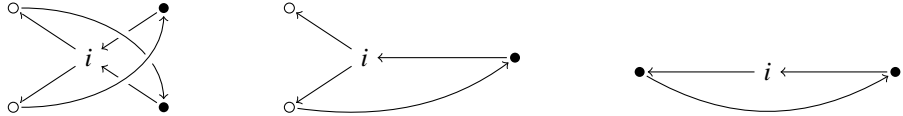
\[\begin{tikzpicture}[baseline=0,yscale=.66]
\node[wert,name=A] at (0,0) {};
\node[vert,name=B] at (2,0) {};
\node[wert,name=C] at (0,-2) {};
\node[vert,name=D] at (2,-2) {};
\node[name=E] at (1,-1) {$i$};
\path[->] (A) edge (E) (C) edge (E)  (E) edge (B) edge (D) (B) edge (A) 
    (D) edge (C);
\end{tikzpicture}\hspace{.6in}
\begin{tikzpicture}[baseline=0,yscale=.66]
\node[wert,name=A] at (0,0) {};
\node[name=B] at (1,-1) {$i$};
\node[vert,name=C] at (3,-1) {};
\node[wert,name=D] at (0,-2) {};
\path[->] (A) edge (B)   (D) edge (B)    (B) edge (C) (C) edge[bend right] (A);
\end{tikzpicture}\hspace{.6in}
\begin{tikzpicture}[baseline=0,scale=.85]
\node[vert,name=A] at (0,-1) {};
\node[name=B] at (2,-1) {$i$};
\node[vert,name=C] at (4,-1) {};
\path[->] (A)  edge (B) (B) edge (C);
\end{tikzpicture}\]
and mutations at $i$
\[\begin{tikzpicture}[baseline=0,yscale=.66]
\node[wert,name=A] at (0,0) {};
\node[vert,name=B] at (2,0) {};
\node[wert,name=C] at (0,-2) {};
\node[vert,name=D] at (2,-2) {};
\node[name=E] at (1,-1) {$i$};
\path[<-] 
	(A) edge (E) (C) edge (E)  (E) edge (B) edge (D)
	(D) edge[preaction=outline,bend right=45] (A)
	(B)  edge[preaction=outline,bend left=45] (C);
\end{tikzpicture}\hspace{.6in}
\begin{tikzpicture}[baseline=0,yscale=.66]
\node[wert,name=A] at (0,0) {};
\node[name=B] at (1,-1) {$i$};
\node[vert,name=C] at (3,-1) {};
\node[wert,name=D] at (0,-2) {};
\path[->] (B) edge (A)   (B) edge (D)    (C) edge (B) (D) edge[bend right] (C);
\end{tikzpicture}\hspace{.6in}
\begin{tikzpicture}[baseline=0,scale=.85]
\node[vert,name=A] at (0,-1) {};
\node[name=B] at (2,-1) {$i$};
\node[vert,name=C] at (4,-1) {};
\path[->] (B)  edge (A) (C) edge (B) (A) edge[bend right] (C);
\end{tikzpicture}\]
In each of these cases it is clear that if we cut $\mu_i(\tilde Q)$ at the
arrows $\alpha$ such that $t(\alpha) = i$, we will recover $R_i(Q)$. 

We now consider those configurations in which $i$ is the target of a relation.
There are five such configurations in which we may reflect at $i$. In these
cases we must consider a local picture that is a two vertex neighborhood of $i$.
First consider those configurations when $i$ is a sink.
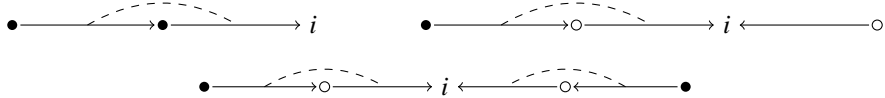
\[\begin{tikzpicture}
\node[vert,name=A] at (0,0) {};
\node[vert,name=B] at (2,0) {};
\node[name=X] at (4,0) {$i$};
\path[->] (A) edge (B) (B) edge (X);
\path[dashed] ($(A)!.5!(B)$) edge[bend left] ($(B)!.5!(X)$);
\end{tikzpicture}\hspace{.5in}
\begin{tikzpicture}
\node[vert,name=A] at (0,0) {};
\node[wert,name=B] at (2,0) {};
\node[name=X] at (4,0) {$i$};
\node[wert,name=C] at (6,0) {};
\path[->] (A) edge (B) (B) edge (X) (C) edge (X) ;
\path[dashed] ($(A)!.5!(B)$) edge[bend left] ($(B)!.5!(X)$);
\end{tikzpicture}\]
\[
\begin{tikzpicture}[scale=.8]
\node[vert,name=A] at (0,0) {};
\node[wert,name=B] at (2,0) {};
\node[name=X] at (4,0) {$i$};
\node[wert,name=C] at (6,0) {};
\node[vert,name=D] at (8,0) {};
\path[->] (A) edge (B) (B) edge (X) (C) edge (X) (D) edge (C);
\path[dashed] 
	($(A)!.5!(B)$) edge[bend left] ($(B)!.5!(X)$)
	($(D)!.5!(C)$) edge[bend right] ($(C)!.5!(X)$);
\end{tikzpicture}\]
The reflection at $i$ for each configuration is
 
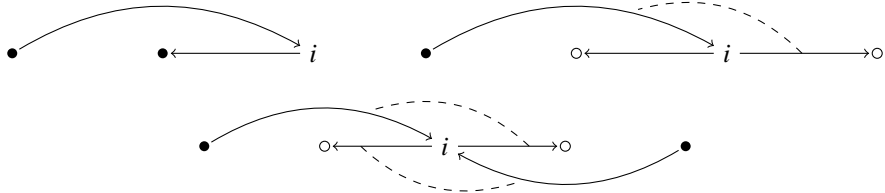
\[\begin{tikzpicture}
\node[vert,name=A] at (0,0) {};
\node[vert,name=B] at (2,0) {};
\node[name=X] at (4,0) {$i$};
\path[->] (A) edge[bend left] node[pos=.7,name=R] {}  (X) (X) edge (B);
\end{tikzpicture}\hspace{.5in}
\begin{tikzpicture}
\node[vert,name=A] at (0,0) {};
\node[wert,name=B] at (2,0) {};
\node[name=X] at (4,0) {$i$};
\node[wert,name=C] at (6,0) {};
\path[->] (A) edge[bend left] node[pos=.7,name=R] {} (X) (X) edge (B) edge (C) ;
\path[dashed] ($(X)!.5!(C)$) edge[bend right] (R);
\end{tikzpicture}\]
\[
\begin{tikzpicture}[scale=.8]
\node[vert,name=A] at (0,0) {};
\node[wert,name=B] at (2,0) {};
\node[name=X] at (4,0) {$i$};
\node[wert,name=C] at (6,0) {};
\node[vert,name=D] at (8,0) {};
\path[->] 
	(A) edge[bend left]  node[pos=.7,name=R1] {} (X) 
	(X) edge (B) edge (C) 
	(D) edge[bend left] node[pos=.7,name=R2] {} (X);
\path[dashed] 
	($(X)!.7!(C)$) edge[bend right] (R1)
	($(X)!.7!(B)$) edge[bend right] (R2);
\end{tikzpicture}\]
The quivers $\tilde Q$ corresponding to $Q$ before the reflection are
 
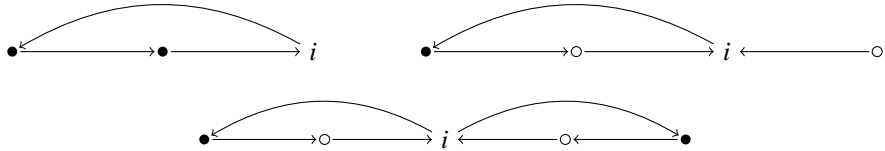
\[\begin{tikzpicture}
\node[vert,name=A] at (0,0) {};
\node[vert,name=B] at (2,0) {};
\node[name=X] at (4,0) {$i$};
\path[->] (A) edge (B) (B) edge (X) (X) edge[bend right] (A);
\end{tikzpicture}\hspace{.5in}
\begin{tikzpicture}
\node[vert,name=A] at (0,0) {};
\node[wert,name=B] at (2,0) {};
\node[name=X] at (4,0) {$i$};
\node[wert,name=C] at (6,0) {};
\path[->] (A) edge (B) (B) edge (X) (C) edge (X)  (X) edge[bend right] (A);
\end{tikzpicture}\]
\[
\begin{tikzpicture}[scale=.8]
\node[vert,name=A] at (0,0) {};
\node[wert,name=B] at (2,0) {};
\node[name=X] at (4,0) {$i$};
\node[wert,name=C] at (6,0) {};
\node[vert,name=D] at (8,0) {};
\path[->] (A) edge (B) (B) edge (X) (C) edge (X) (D) edge (C) 
    (X) edge[bend right] (A) edge[bend left] (D);
\end{tikzpicture}\]
The mutation at $i$ gives
 
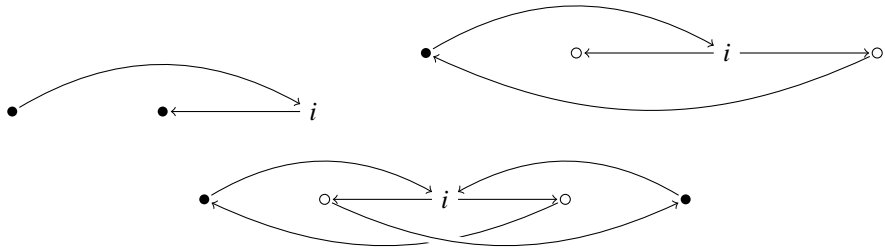
\[\begin{tikzpicture}
\node[vert,name=A] at (0,0) {};
\node[vert,name=B] at (2,0) {};
\node[name=X] at (4,0) {$i$};
\path[->] (X) edge (B) (A) edge[bend left] (X);
\end{tikzpicture}\hspace{.5in}
\begin{tikzpicture}
\node[vert,name=A] at (0,0) {};
\node[wert,name=B] at (2,0) {};
\node[name=X] at (4,0) {$i$};
\node[wert,name=C] at (6,0) {};
\path[->] 
	(X) edge (B) (A) edge[bend left] (X) (X) edge (C)  
	(C) edge[bend left=25] (A);
\end{tikzpicture}\]
\[
\begin{tikzpicture}[scale=.8]
\node[vert,name=A] at (0,0) {};
\node[wert,name=B] at (2,0) {};
\node[name=X] at (4,0) {$i$};
\node[wert,name=C] at (6,0) {};
\node[vert,name=D] at (8,0) {};
\path[->] 
	(C) edge[bend left=25] (A)
	(B) edge[preaction=outline,bend right=25] (D)
	(X) edge (B) (A) edge[bend left] (X) (X) edge (C)  (D) edge[bend right] (X);
\end{tikzpicture}\]
In each of these local configurations, if we cut the arrow(s) $\alpha$ with
$s(\alpha)\neq i\neq t(\alpha)$, then we recover $R_i(Q)$. 

If $i$ is neither a source nor a sink, and we may reflect at $i$, then we have
one of the following local configurations  
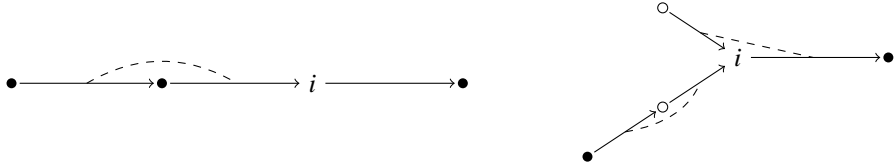
\[\begin{tikzpicture}[baseline=1cm]
\node[vert,name=A] at (0,0) {};
\node[vert,name=B] at (2,0) {};
\node[name=X] at (4,0) {$i$};
\node[vert,name=C] at (6,0) {};
\path[->] (A) edge (B) (B) edge (X) (X) edge (C) ;
\path[dashed] ($(A)!.5!(B)$) edge[bend left] ($(B)!.5!(X)$);
\end{tikzpicture}\hspace{.6in}
\begin{tikzpicture}[baseline=0,yscale=.66]
\node[wert,name=A] at (0,0) {};
\node[name=B] at (1,-1) {$i$};
\node[vert,name=C] at (3,-1) {};
\node[wert,name=D] at (0,-2) {};
\node[vert,name=E] at (-1,-3){};
\path[->] (A) edge (B)   (D) edge (B)    (B) edge (C) (E) edge (D);
\path[dashed] ($(A)!.5!(B)$) edge ($(B)!.5!(C)$) 
    ($(E)!.5!(D)$) edge[bend right] ($(D)!.5!(B)$);
\end{tikzpicture}\]
The reflections at $i$ are
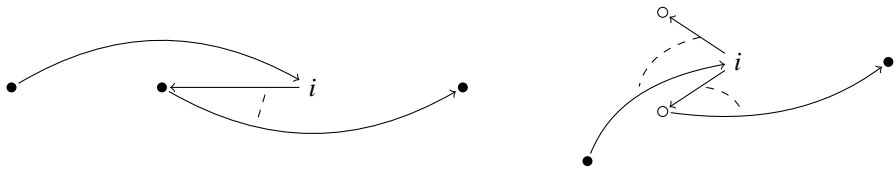
\[\begin{tikzpicture}[baseline=1cm]
\node[vert,name=A] at (0,0) {};
\node[vert,name=B] at (2,0) {};
\node[name=X] at (4,0) {$i$};
\node[vert,name=C] at (6,0) {};
\path[->] (A) edge[bend left] (X) (X) edge (B) 
    (B) edge[bend right] node[name=R,pos=.3] {} (C) ;
\path[dashed] (R) edge ($(B)!.7!(X)$);
\end{tikzpicture}\hspace{.6in} 
\begin{tikzpicture}[baseline=0,yscale=.66]
\node[wert,name=A] at (0,0) {};
\node[name=X] at (1,-1) {$i$};
\node[vert,name=C] at (3,-1) {};
\node[wert,name=D] at (0,-2) {};
\node[vert,name=E] at (-1,-3){};
\path[->] 
	(X) edge (A)  edge (D) 
	(D) edge[bend right]  node[pos=.3,name=R1] {} (C) 
	(E) edge[bend left] node[pos=.5,name=R2]{} (X);
\path[dashed] ($(A)!.5!(X)$) edge[bend right] (R2) 
    (R1) edge[bend right] ($(D)!.5!(X)$);
\end{tikzpicture}\]
The quivers $\tilde Q$ corresponding to $Q$ before the reflection are
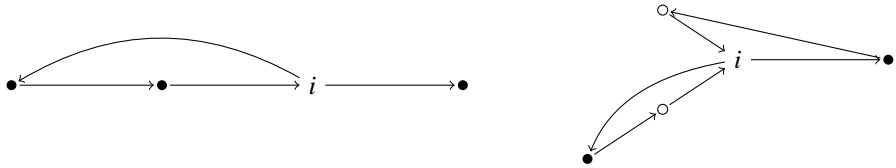
\[\begin{tikzpicture}[baseline=1cm]
\node[vert,name=A] at (0,0) {};
\node[vert,name=B] at (2,0) {};
\node[name=X] at (4,0) {$i$};
\node[vert,name=C] at (6,0) {};
\path[->] (A) edge (B) (B) edge (X) (X) edge (C) (X) edge[bend right] (A);
\end{tikzpicture}\hspace{.6in}
\begin{tikzpicture}[baseline=0,yscale=.66]
\node[wert,name=A] at (0,0) {};
\node[name=X] at (1,-1) {$i$};
\node[vert,name=C] at (3,-1) {};
\node[wert,name=D] at (0,-2) {};
\node[vert,name=E] at (-1,-3){};
\path[->] 
    (A) edge (X)   (D) edge (X)    (X) edge (C) (E) edge (D) 
    (C) edge (A) (X) edge[bend right] (E);
\end{tikzpicture}\]
The mutations at $i$ are 
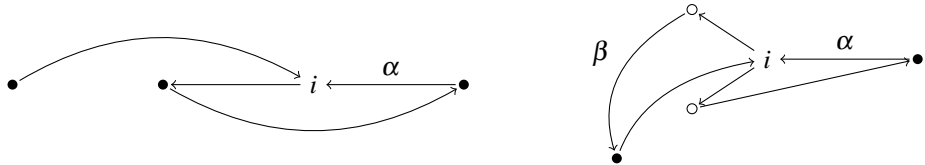
\[\begin{tikzpicture}[baseline=1cm]
\node[vert,name=A] at (0,0) {};
\node[vert,name=B] at (2,0) {};
\node[name=X] at (4,0) {$i$};
\node[vert,name=C] at (6,0) {};
\path[->] (X) edge (B) (C) edge  node[auto,swap] {$\alpha$}  (X) 
    (A) edge[bend left] (X) (B) edge[bend right] (C); 
\end{tikzpicture}\hspace{.6in}
\begin{tikzpicture}[baseline=0,yscale=.66]
\node[wert,name=A] at (0,0) {};
\node[name=B] at (1,-1) {$i$};
\node[vert,name=C] at (3,-1) {};
\node[wert,name=D] at (0,-2) {};
\node[vert,name=E] at (-1,-3){};
\path[->] (B) edge (A)   (B) edge (D)    
    (C) edge  node[auto,swap] {$\alpha$}  (B)  
	(E) edge[bend left]  (B) 
	(A) edge[bend right] node[auto,swap] {$\beta$} (E) 
	(D) edge (C) ;
\end{tikzpicture}\]
In the first case we recover $R_i(Q)$ by cutting the arrow $\alpha$ with
$t(\alpha)=i$. Note that this is well-defined because there is only one cycle. In
the second case we must cut the two arrows marked $\alpha$ and $\beta$ in the
diagram.
\end{proof}

Unfortunately, this type of proof doe not really explain what is happening.  
The connection with mutation becomes more when we translate the above dictionary into the 
cut surface. Like cluster mutations, we can express reflections as an operation on the edges
in the triangulation of $(S,M,T)$. we list a local configuration at a vertex $i$
and corresponding local picture in $(S,M,T)$. The corresponding reflection at
$i$ is given to the right. The red lines represent which vertices are cut,
the line passing between edges $i$ and $j$ represents either $\chi_{i,j}$ or 
$\chi_{j,i}$ depending on the orientation of triangle.  As in the proof
of Theorem~\ref{prop dict} we do not include pictures for those 
configurations with double arrows, those are `degenerate' cases of the pictures
given.

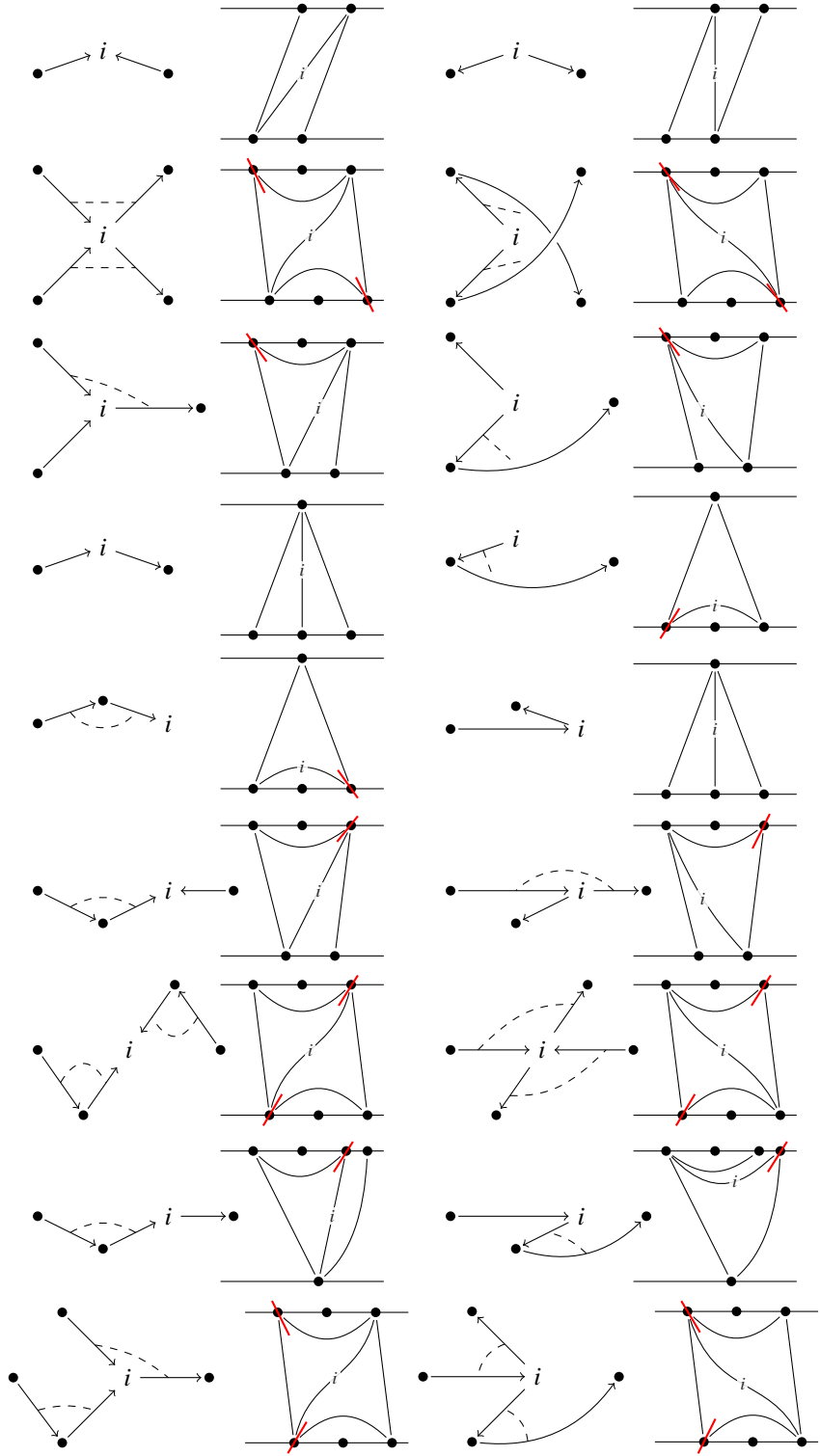
\begin{center}
\renewcommand{\arraystretch}{1.5}
\begin{longtable}{c @{\hspace{.05in}} c}
 $Q$  & $R_i(Q)$ \\
\begin{tikzpicture}[scale=.9]
{[]
	\node[vert,name=A] at (0,-1) {};
	\node[name=B] at (1,-.65) {$i$};
	\node[vert,name=C] at (2,-1) {};
	\path[->] (A)  edge (B) (C) edge (B);
}
 { [xshift=3.3cm,xscale=.5]
	\draw (-1,0) -- (4,0)  (-1,-2) -- (4,-2);
	\foreach \x [count=\n]  in {1.5,3}
		\node[vert,name=t\n] at (\x,0) {};
	\foreach \y [count=\n] in {.5,2}
		\node[vert,name=b\n] at (\y-.5,-2) {};
	\draw (t1) edge (b1) 
	(b1) edge node[int] {$i$} (t2)
	(t2) edge (b2);
}
\end{tikzpicture} 
& \begin{tikzpicture}[scale=.9]
{[]
	\node[vert,name=A] at (0,-1) {};
	\node[name=B] at (1,-.65) {$i$};
	\node[vert,name=C] at (2,-1) {};
	\path[->] (B) edge (C) edge (A);
}
 { [xshift=3.3cm,xscale=.5]
	\draw (-1,0) -- (4,0)  (-1,-2) -- (4,-2);
	\foreach \x [count=\n]  in {1.5,3}
		\node[vert,name=t\n] at (\x,0) {};
	\foreach \y [count=\n] in {.5,2}
		\node[vert,name=b\n] at (\y-.5,-2) {};
	\draw (t1) edge (b1)  edge node[int] {$i$} (b2)
	(t2) edge (b2);
}
\end{tikzpicture}\\
\begin{tikzpicture}[scale=.9]
{ []
	\node[vert,name=A] at (0,0) {};
	\node[vert,name=B] at (2,0) {};
	\node[vert,name=C] at (0,-2) {};
	\node[vert,name=D] at (2,-2) {};
	\node[name=E] at (1,-1) {$i$};
	\path[->] (A) edge (E)   (C) edge (E)    (E) edge (B) edge (D);
	\path[dashed] ($(A)!.5!(E)$) edge ($(E)!.5!(B)$) ($(C)!.5!(E)$) edge ($(E)!.5!(D)$);
}
{ [xshift=3.3cm,xscale=.5]
	\draw (-1,0) -- (4,0)  (-1,-2) -- (4,-2);
	\foreach \x [count=\n]  in {0,1.5,3}
		\node[vert,name=T\n] at (\x,0) {};
	\foreach \y [count=\n] in {0,1.5,3}
		\node[vert,name=B\n] at (\y+.5,-2) {};
	\draw (T1) edge[bend right]  (T3) edge (B1)
		(T3) edge (B3)
		(B1) edge[bend left] (B3);
	\draw (T3) .. controls ($(T2)!.5!(B3)$) and ($(T1)!.5!(B2)$) .. (B1) node[int] {$i$};
	\draw[rcut] ($(T1)+(135:.25cm)$) -- ($(T1)+(315:.5cm)$) 
		 ($(B3)+(135:.5cm)$) -- ($(B3)+(315:.25cm)$) ;
}
\end{tikzpicture} 
& \begin{tikzpicture}[scale=.9]
{ []
	\node[vert,name=A] at (0,0) {};
	\node[vert,name=B] at (2,0) {};
	\node[vert,name=C] at (0,-2) {};
	\node[vert,name=D] at (2,-2) {};
	\node[name=E] at (1,-1) {$i$};
	\path[->]  
	   (A) edge[bend left] node[name=R1,pos=.5] {} (D)  (C) edge[preaction=outline,bend right]  node[name=R2,pos=.5] {} (B)  
	   (E) edge (A) edge (C);
	\path[dashed] ($(E)!.5!(A)$) edge (R1) ($(E)!.5!(C)$) edge (R2);
}
{ [xshift=3.3cm,xscale=.5]
	\draw (-1,0) -- (4,0)  (-1,-2) -- (4,-2);
	\foreach \x [count=\n]  in {0,1.5,3}
		\node[vert,name=T\n] at (\x,0) {};
	\foreach \y [count=\n] in {0,1.5,3}
		\node[vert,name=B\n] at (\y+.5,-2) {};
	\draw (T1) edge[bend right]  (T3) edge (B1)
		(T3) edge (B3)
		(B1) edge[bend left] (B3);
	\draw (T1) .. controls ($(T2)!.5!(B1)$) and ($(T3)!.5!(B2)$) .. (B3) node[int] {$i$};
	\draw[rcut] ($(T1)+(145:.25cm)$) -- ($(T1)+(325:.5cm)$) 
		 ($(B3)+(145:.5cm)$) -- ($(B3)+(325:.25cm)$) ;
}
\end{tikzpicture}\\
\begin{tikzpicture}[scale=.9]
{ []
	\node[vert,name=A] at (0,0) {};
	\node[name=B] at (1,-1) {$i$};
	\node[vert,name=C] at (2.5,-1) {};
	\node[vert,name=D] at (0,-2) {};
	\path[->] (A) edge (B)   (D) edge (B)    (B) edge (C);
	\path[dashed] ($(A)!.5!(B)$) edge[bend left=10] ($(B)!.5!(C)$);
}
{ [xshift=3.3cm,xscale=.5]
	\draw (-1,0) -- (4,0)  (-1,-2) -- (4,-2);
	\foreach \x [count=\n]  in {0,1.5,3}
		\node[vert,name=T\n] at (\x,0) {};
	\foreach \y [count=\n] in {.5,2}
		\node[vert,name=B\n] at (\y+.5,-2) {};
	\draw (T1) edge[bend right=20]  (T3) edge (B1)
		(T3) edge node[int] {$i$} (B1) edge (B2);
	\draw[rcut] ($(T1)+(145:.25cm)$) -- ($(T1)+(325:.5cm)$);
}
\end{tikzpicture}
& \begin{tikzpicture}[scale=.9]
 { []
	 \node[vert,name=A] at (0,0) {};
	\node[name=B] at (1,-1) {$i$};
	\node[vert,name=C] at (2.5,-1) {};
	\node[vert,name=D] at (0,-2) {};
	\path[->] (B) edge (A) edge (D) (D) edge[bend right] node[pos=.35,name=R] {} (C);
	\path[dashed] ($(B)!.5!(D)$) edge (R);
 }
 { [xshift=3.3cm,xscale=.5]
	\draw (-1,0) -- (4,0)  (-1,-2) -- (4,-2);
	\foreach \x [count=\n]  in {0,1.5,3}
		\node[vert,name=t\n] at (\x,0) {};
	\foreach \y [count=\n] in {.5,2}
		\node[vert,name=b\n] at (\y+.5,-2) {};
	\draw (t1) edge[bend right=20]  (t3) edge (b1)
		(t1) edge[bend right=10] node[int] {$i$} (b2) 
		(t3) edge (b2);
	\draw[rcut] ($(t1)+(145:.25cm)$) -- ($(t1)+(325:.5cm)$);
}
\end{tikzpicture}\\
\begin{tikzpicture}[scale=.9]
{[]
	\node[vert,name=A] at (0,-1) {};
	\node[name=B] at (1,-.65) {$i$};
	\node[vert,name=C] at (2,-1) {};
	\path[->] (A)  edge (B) (B) edge (C);
}
 { [xshift=3.3cm,xscale=.5]
	\draw (-1,0) -- (4,0)  (-1,-2) -- (4,-2);
	\foreach \x [count=\n]  in {1.5}
		\node[vert,name=t\n] at (\x,0) {};
	\foreach \y [count=\n] in {.5,2,3.5}
		\node[vert,name=b\n] at (\y-.5,-2) {};
	\draw (t1) edge (b1) edge node[int] {$i$} (b2) edge (b3);
}
\end{tikzpicture}
& \begin{tikzpicture}[scale=.9]
{[]
	\node[vert,name=A] at (0,-1) {};
	\node[name=B] at (1,-.65) {$i$};
	\node[vert,name=C] at (2.5,-1) {};
	\path[->] (B)  edge (A) (A) edge[bend right] node[pos=.25,name=R] {} (C);
	\path[dashed] ($(B)!.5!(A)$) edge (R);
}
 { [xshift=3.3cm,xscale=.5]
	\draw (-1,0) -- (4,0)  (-1,-2) -- (4,-2);
	\foreach \x [count=\n]  in {1.5}
		\node[vert,name=t\n] at (\x,0) {};
	\foreach \y [count=\n] in {.5,2,3.5}
		\node[vert,name=b\n] at (\y-.5,-2) {};
	\draw (t1) edge (b1) edge (b3)
	(b1) edge[bend left=20] node[int] {$i$} (b3);
	\draw[rcut] ($(b1)+(35:.5cm)$) -- ($(b1)+(225:.25cm)$);
}
\end{tikzpicture}\\

\begin{tikzpicture}[scale=.9]
{ []
	\node[vert,name=A] at (0,-1) {};
	\node[vert,name=B] at (1,-.65) {};
	\node[name=X] at (2,-1) {$i$};
	\path[->] (A) edge (B) (B) edge (X);
	\path[dashed] ($(A)!.5!(B)$) edge[bend right=55] ($(B)!.5!(X)$);
}
 { [xshift=3.3cm,xscale=.5]
	\draw (-1,0) -- (4,0)  (-1,-2) -- (4,-2);
	\foreach \x [count=\n]  in {1.5}
		\node[vert,name=t\n] at (\x,0) {};
	\foreach \y [count=\n] in {.5,2,3.5}
		\node[vert,name=b\n] at (\y-.5,-2) {};
	\draw (t1) edge (b1) edge (b3)
	(b1) edge[bend left=20] node[int] {$i$} (b3);
	\draw[rcut] ($(b3)+(145:.5cm)$) -- ($(b3)+(-35:.25cm)$);
}
\end{tikzpicture}
& \begin{tikzpicture}[scale=.9]
{[]
\node[vert,name=A] at (0,-1) {};
\node[vert,name=B] at (1,-.65) {};
\node[name=X] at (2,-1) {$i$};
\path[->] (A) edge node[pos=.7,name=R] {}  (X) (X) edge (B);
}
 { [xshift=3.3cm,xscale=.5]
	\draw (-1,0) -- (4,0)  (-1,-2) -- (4,-2);
	\foreach \x [count=\n]  in {1.5}
		\node[vert,name=t\n] at (\x,0) {};
	\foreach \y [count=\n] in {.5,2,3.5}
		\node[vert,name=b\n] at (\y-.5,-2) {};
	\draw (t1) edge (b1) edge node[int] {$i$} (b2) edge (b3);
}
\end{tikzpicture} \\
\begin{tikzpicture}[scale=.9]
{[]
	\node[vert,name=A] at (0,-1) {};
	\node[vert,name=B] at (1,-1.5) {};
	\node[name=X] at (2,-1) {$i$};
	\node[vert,name=C] at (3,-1) {};
	\path[->] (A) edge (B) (B) edge (X) (C) edge (X) ;
	\path[dashed] ($(A)!.5!(B)$) edge[bend left] ($(B)!.5!(X)$);
}
{ [xshift=3.3cm,xscale=.5]
	\draw (-1,0) -- (4,0)  (-1,-2) -- (4,-2);
	\foreach \x [count=\n]  in {0,1.5,3}
		\node[vert,name=T\n] at (\x,0) {};
	\foreach \y [count=\n] in {.5,2}
		\node[vert,name=B\n] at (\y+.5,-2) {};
	\draw (T1) edge[bend right=20]  (T3) edge (B1)
		(T3) edge node[int] {$i$} (B1) edge (B2);
	\draw[rcut] ($(T3)+(35:.25cm)$) -- ($(T3)+(-150:.5cm)$);
}
\end{tikzpicture}
&\begin{tikzpicture}[scale=.9]
{[yshift=-1cm]
	\node[vert,name=A] at (0,0) {};
	\node[vert,name=B] at (1,-.5) {};
	\node[name=X] at (2,0) {$i$};
	\node[vert,name=C] at (3,0) {};
	\path[->] (A) edge (X) (X) edge (B) edge (C) ;
	\path[dashed] ($(X)!.5!(C)$) edge[bend right=45] ($(X)!.5!(A)$) ;
}
 { [xshift=3.3cm,xscale=.5]
	\draw (-1,0) -- (4,0)  (-1,-2) -- (4,-2);
	\foreach \x [count=\n]  in {0,1.5,3}
		\node[vert,name=t\n] at (\x,0) {};
	\foreach \y [count=\n] in {.5,2}
		\node[vert,name=b\n] at (\y+.5,-2) {};
	\draw (t1) edge[bend right=20]  (t3) edge (b1)
		(t1) edge[bend right=10] node[int] {$i$} (b2) 
		(t3) edge (b2);
	\draw[rcut] ($(t3)+(45:.25cm)$) -- ($(t3)+(-135:.5cm)$);
} 
\end{tikzpicture} \\
\begin{tikzpicture}[scale=.9]
{[yshift=-1cm,xscale=.35]
	\node[vert,name=A] at (0,0) {};
	\node[vert,name=B] at (2,-1) {};
	\node[name=X] at (4,0) {$i$};
	\node[vert,name=C] at (6,1) {};
	\node[vert,name=D] at (8,0) {};
	\path[->] (A) edge (B) (B) edge (X) (C) edge (X) (D) edge (C);
	\path[dashed] 
		($(A)!.5!(B)$) edge[bend left] ($(B)!.5!(X)$)
		($(D)!.5!(C)$) edge[bend left] ($(C)!.5!(X)$);
}
{ [xshift=3.3cm,xscale=.5]
	\draw (-1,0) -- (4,0)  (-1,-2) -- (4,-2);
	\foreach \x [count=\n]  in {0,1.5,3}
		\node[vert,name=t\n] at (\x,0) {};
	\foreach \y [count=\n] in {0,1.5,3}
		\node[vert,name=b\n] at (\y+.5,-2) {};
	\draw (t1) edge[bend right=25]  (t3) edge (b1)
		(t3) edge (b3)
		(b1) edge[bend left=25] (b3);
	\draw (t3) .. controls ($(t2)!.5!(b3)$) and ($(t1)!.5!(b2)$) .. (b1) node[int] {$i$};
	\draw[rcut] ($(t3)+(35:.25cm)$) -- ($(t3)+(-140:.5cm)$) 
		       ($(b1)+(40:.5cm)$) -- ($(b1)+(-140:.25cm)$) ;
}
\end{tikzpicture} 
&\begin{tikzpicture}[scale=.9]
{[yshift=-1cm,xscale=.35]
\node[vert,name=A] at (0,0) {};
\node[vert,name=B] at (2,-1) {};
\node[name=X] at (4,0) {$i$};
\node[vert,name=C] at (6,1) {};
\node[vert,name=D] at (8,0) {};
\path[->] 
	(A) edge (X) 
	(X) edge (B) edge (C) 
	(D) edge (X);
\path[dashed] 
	($(X)!.7!(A)$) edge[bend left=10] ($(X)!.7!(C)$)
	($(X)!.7!(B)$) edge[bend right=10] ($(X)!.7!(D)$);
}
{ [xshift=3.3cm,xscale=.5]
	\draw (-1,0) -- (4,0)  (-1,-2) -- (4,-2);
	\foreach \x [count=\n]  in {0,1.5,3}
		\node[vert,name=t\n] at (\x,0) {};
	\foreach \y [count=\n] in {0,1.5,3}
		\node[vert,name=b\n] at (\y+.5,-2) {};
	\draw (t1) edge[bend right=25]  (t3) edge (b1)
		(t3) edge (b3)
		(b1) edge[bend left=25] (b3);
	\draw (t1) .. controls ($(t2)!.5!(b1)$) and ($(t3)!.5!(b2)$) .. (b3) node[int] {$i$};
	\draw[rcut] ($(t3)+(35:.25cm)$) -- ($(t3)+(-140:.5cm)$) 
		       ($(b1)+(40:.5cm)$) -- ($(b1)+(-140:.25cm)$) ;
}
\end{tikzpicture} \\
\begin{tikzpicture}[scale=.9]
{[yshift=-1cm]
	\node[vert,name=A] at (0,0) {};
	\node[vert,name=B] at (1,-.5) {};
	\node[name=X] at (2,0) {$i$};
	\node[vert,name=C] at (3,0) {};
	\path[->] (A) edge (B) (B) edge (X) (X) edge (C) ;
	\path[dashed] ($(A)!.5!(B)$) edge[bend left] ($(B)!.5!(X)$);
}
{ [xshift=3.3cm,xscale=.5]
	\draw (-1,0) -- (4,0)  (-1,-2) -- (4,-2);
	\foreach \x [count=\n]  in {0,1.5,2.85,3.5}
		\node[vert,name=t\n] at (\x,0) {};
	\foreach \y [count=\n] in {1.5}
		\node[vert,name=b\n] at (\y+.5,-2) {};
	\draw (t1) edge[bend right=25]  (t3) edge (b1)
		(t3) edge node[int] {$i$} (b1)
		(b1) edge[bend right] (t4);
	\draw[rcut] ($(t3)+(35:.25cm)$) -- ($(t3)+(-140:.5cm)$) ;
}
\end{tikzpicture}
&\begin{tikzpicture}[scale=.9]
{[yshift=-1cm]
	\node[vert,name=A] at (0,0) {};
	\node[vert,name=B] at (1,-.5) {};
	\node[name=X] at (2,0) {$i$};
	\node[vert,name=C] at (3,0) {};
	\path[->] (A) edge(X) 
		(X) edge (B)
		(B) edge[bend right] node[coordinate,name=R] {}  (C) ;
	\path[dashed] (R) edge[bend right=15] ($(B)!.5!(X)$);
}
{ [xshift=3.3cm,xscale=.5]
	\draw (-1,0) -- (4,0)  (-1,-2) -- (4,-2);
	\foreach \x [count=\n]  in {0,1.5,2.85,3.5}
		\node[vert,name=t\n] at (\x,0) {};
	\foreach \y [count=\n] in {1.5}
		\node[vert,name=b\n] at (\y+.5,-2) {};
	\draw (t1) edge[bend right=20]  (t3) edge (b1)
		(t1) .. controls ($(t2)!.45!(b1)$) and ($(t3)!.15!(b1)$) .. (t4) node[int,pos=.5] {$i$} 
		(b1) edge[bend right] (t4);
	\draw[rcut] ($(t4)+(35:.25cm)$) -- ($(t4)+(-140:.5cm)$) ;
}
\end{tikzpicture} \\
\begin{tikzpicture}[scale=.9]
{[yshift=0cm]
	\node[vert,name=A] at (0,0) {};
	\node[name=B] at (1,-1) {$i$};
	\node[vert,name=C] at (2.25,-1) {};
	\node[vert,name=D] at (0,-2) {};
	\node[vert,name=E] at (-.75,-1) {};
	\path[->] (A) edge (B)   (D) edge (B)    (B) edge (C) (E) edge (D);
	\path[dashed] ($(A)!.5!(B)$) edge[bend left=15] ($(B)!.5!(C)$) 
			  ($(E)!.5!(D)$) edge[bend left=15] ($(D)!.5!(B)$);
}
{ [xshift=3.3cm,xscale=.5]
	\draw (-1,0) -- (4,0)  (-1,-2) -- (4,-2);
	\foreach \x [count=\n]  in {0,1.5,3}
		\node[vert,name=t\n] at (\x,0) {};
	\foreach \y [count=\n] in {0,1.5,3}
		\node[vert,name=b\n] at (\y+.5,-2) {};
	\draw (t1) edge[bend right=25]  (t3) edge (b1)
		(t3) edge (b3)
		(b1) edge[bend left=25] (b3);
	\draw (t3) .. controls ($(t2)!.5!(b3)$) and ($(t1)!.5!(b2)$) .. (b1) node[int] {$i$};
	\draw[rcut] ($(t1)+(135:.25cm)$) -- ($(t1)+(-45:.5cm)$) 
		       ($(b1)+(40:.5cm)$) -- ($(b1)+(-140:.25cm)$) ;
}
\end{tikzpicture}
&\begin{tikzpicture}[scale=.9]
{ []
	\node[vert,name=A] at (0,0) {};
	\node[name=X] at (1,-1) {$i$};
	\node[vert,name=C] at (2.25,-1) {};
	\node[vert,name=D] at (0,-2) {};
	\node[vert,name=E] at (-.75,-1) {};
	\path[->] 
		(X) edge (A)  edge (D) 
		(D) edge[bend right]  node[coordinate,pos=.3,name=R1] {} (C) 
		(E) edge node[coordinate,pos=.5,name=R2]{} (X);
	\path[dashed] ($(A)!.5!(X)$) edge[bend right] (R2) (R1) edge[bend right] ($(D)!.5!(X)$);
}
{ [xshift=3.3cm,xscale=.5]
	\draw (-1,0) -- (4,0)  (-1,-2) -- (4,-2);
	\foreach \x [count=\n]  in {0,1.5,3}
		\node[vert,name=t\n] at (\x,0) {};
	\foreach \y [count=\n] in {0,1.5,3}
		\node[vert,name=b\n] at (\y+.5,-2) {};
	\draw (t1) edge[bend right=25]  (t3) edge (b1)
		(t3) edge (b3)
		(b1) edge[bend left=25] (b3);
	\draw (t1) .. controls ($(t2)!.65!(b1)$) and ($(t2)!.5!(b3)$) .. (b3) node[int] {$i$};
	\draw[rcut] ($(t1)+(140:.25cm)$) -- ($(t1)+(-40:.5cm)$) 
		       ($(b1)+(45:.5cm)$) -- ($(b1)+(-135:.25cm)$) ;
}
\end{tikzpicture}

\end{longtable}
\end{center}

\begin{dfn}
Let $(S,M,T)$ be a triangulated surface and $\te{}\in T$ the diagonal of a 
rectangle with vertices $abcd$ such that the endpoints of $\te{}$ are at $b$ and $d$.
We define $\te{}^\perp$ to be the arc that is the other diagonal of $abcd$. A
\df{clockwise twist} of $\te{}$ is an free isotopy $\Phi\colon S\times[0,1] \to S$
with $\te{}^\perp$ such that the endpoints of $\Phi(\te{},t)$ are contained in the
edges $bc$ and $ad$ for each $t$. Similarly, a \df{counterclockwise twist} is
given by a free isotopy $\Phi$ such that the endpoints of $\Phi(\te{},t)$ are
contained in the edges $ab$ and $cd$.
\end{dfn}

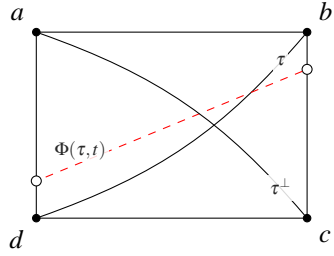
\begin{figure}[h]
\centering
\begin{tikzpicture}[scale=.9]
\draw (0,0) -- (4,0) node[pos=0,solid,name=u,label=135:$a$] {}
            -- (4,-2.75) node[pos=0,solid,name=v,label=45:$b$] {} node[pos=.2,name=v',shape=circle,draw,scale=.4,fill=white] {}
            -- (0,-2.75) node[pos=0,solid,name=x,label=-45:$c$] {}
            -- (0,0) node[pos=0,solid,name=y,label=-135:$d$] {} node[pos=.2,name=y',shape=circle,draw,scale=.4,fill=white] {};
\path (v) edge[bend left=15] node[int,pos=.1] {$\te{}$} (y) ;
\draw[dashed,red] (v') -- (y') node[int,pos=.85,above,text=black] {$\Phi(\te{},t)$};
\path (u) edge[bend left=15] node[int,pos=.9] {$\te{}^\perp$} (x) ;

\end{tikzpicture}
\caption{The clockwise twist of $\te{}$.  The dashed line represents $\Phi$ at time $t$.}\label{fig:twist}
\end{figure}

We can view the twist operation as an operation on the triangulation $(S,M,T)$, 
The twist at $\te{}$ produces a new triangulation $(S,M,T')$ which 
differs from $T$ at only $\te{}$. Depending on the types of edges bounding the 
rectangle containing $\te{}$, the types of triangles defined by $T'$ may be 
different than the types of triangles defined by $T$. For example, if the vertex $i$,
corresponding to $\te i$ is a sink and the end of a relation, then the rectangle 
containing $\te{i}$  has exactly one internal triangle while the rectangle
containing the twist of $\te i$, $\te{i}^\perp$, does not contain an internal triangle,
see the dictonary table above.  

Using the above dictionary we have the following proposition.

\begin{prop}\label{prop:dict}
Let $(S,M,T)$ be a triangulated surface and $\cutt$ an admissible cut of 
$(S,M,T)$.  Let $\te{}$ be an arc of $T$ contained in a rectangle $abcd$ such 
that $\te{}$ is not the source of a relation in $Q_{T\dag}$. Then the (co-)
reflection at $\te{}$ is given by a (counter-) clockwise twist $\Phi$ which does 
not pass through any local cut of $\cutt$. Further, if the twist results in 
at least one internal triangle and
\begin{enum}
\item if the original cut vertices of $abcd$ are still contained in internal triangles,
 the local cuts in the rectangle containing $\te{}$ does not change vertices; or,
\item if the original cut vertex of $abcd$ is no longer contained in an internal triangle,
the new cut is incident to $\Phi(\te{},1)$ at the same endpoint of $\te{}$ as the cut incident 
to $\Phi(\te{},0)$.
\end{enum}

Otherwise the (co-) refection does not result in any internal triangles, so $\cutt$ 
has one less local cut.\qed
\end{prop}

In most cases the local cuts do not change vertices.  A change in the location
of a local cut only occurs when the internal triangle it associated with is
destroyed by the reflection. The reflection need not create a new internal 
triangle, but when it does this new internal triangle will have a local cut. 

\begin{cor}
Let $(S,M,T)$ and $(S,M,T')$ be two triangulations of the same unpunctured surface and
$\alg$ and $\alg'$ surface algebras corresponding to admissible cuts
$\cutt$ and $\cutt'$ respectively.  If there is a sequence of reflections of the type
described in the above dictionary such that $(\cutt,\cutt')$ are equi-distributed, then 
$\alg$ and $\alg'$ are derived equivalent.
\end{cor}

\begin{proof}
This follows immediately from Theorem~\ref{mainthm} and Proposition~\ref{prop:dict}.
\end{proof}

\subsection{Reflections in a strip}  
Throughout the remainder of this section we fix $S$ to be an annulus. We use the above 
dictionary to provide an explicit method to construct a sequence of derived equivalences 
between surface algebras of $S$.  In particular, we re-prove Theorem~\ref{thm1} for the annulus.  
This proof gives a more explicit construction of the derived equivalence in terms
of module categories and tilting than is obtained via the direct application of 
Theorem~\ref{thm1}.
\begin{dfn}\label{df:phi-T}
Let $\chi$ be a cut of the triangulation $(S,M,T)$, $\bcpt$ be a boundary component
of $S$ and $\triangle$  a triangle in $T$.  We set $\cutt(\bcpt)$ to be the number
of local cuts in $\cutt$ on $B$ and $\cutt_{\triangle}(\bcpt)$ the number of local 
cuts in $\cutt$ on $\bcpt$ contained in $\triangle$.  
\end{dfn}

Note that while 
$0\leq\cutt(\bcpt)\leq n$, where $n$ is the number of internal triangles, we always 
have $\cutt_\bcpt(\triangle)$ is either zero or one.

\begin{lemma} Fix a boundary component $\bcpt$ and cuts $\cutt_1$ and $\cutt_2$ 
such that $\cutt_1(\bcpt)=\cutt_2(\bcpt)$.  
Define $D=\{ \triangle: \cutt_{1,\triangle}(\bcpt)\neq \cutt_{2,\triangle}(\bcpt)\}$.  
Then $\#D=2m$ for some $m\in \NN$. Further for each triangle $\triangle\in D$ there 
is a corresponding triangle
$\triangle'$ with $\cutt_{1,\triangle}(B) = \cutt_{2,\triangle'}(B)$.
\end{lemma}
\begin{remark}
Because of the restriction that $\cutt_1(\bcpt)=\cutt_2(\bcpt)$ and that $S$ is the 
annulus, the set $D$ does not depend on $B$.
\end{remark}
\begin{proof}
We claim that we can pair up all of the triangles in $D$, that is there is some 
bijection $D \to D$ with no fixed points such that 
$\chi_{1,\triangle}(B)=\chi_{2,\triangle'}(B)$.  Notice that we can write  
$\cutt_1(\bcpt) = \sum_{\triangle} \cutt_{1,\triangle}(B)$ where we sum over those 
triangles $\triangle$ incident to $B$, similarly for $\cutt_{2}(\bcpt)$. Then we 
have
\[
0 = \chi_1(B) - \chi_2(B) = \sum_{\triangle} \cutt_{1,\triangle}(B) - \cutt_{2,\triangle}(B).
\]
Further, we can restrict the sum to only those triangles in $D$ because we clearly get cancellation 
for those triangles not in $D$. 
\[
0 = \chi_1(B) - \chi_2(B) = \sum_{\triangle\in D} \cutt_{1,\triangle}(B) - \cutt_{2,\triangle}(B).
\]
It follows that for each triangle $\triangle\in D$ there is a distinct corresponding
$\triangle'$ with $\cutt_{1,\triangle}(B) = \cutt_{2,\triangle'}(B)$ and hence $\#D = 2m$ for some
$m\in \NN$. 
\end{proof}

In the subsequent lemmas we assume that the algebras $\alg_1$ and $\alg_2$ come 
from admissible cuts $\cutt_1$ and $\cutt_2$ respectively.  We further assume that 
$\#D =2$ and $\cutt_1(\bcpt) = \cutt_2(\bcpt)$. 
Set $D=\{\triangle_1,\triangle_2\}$. These lemmas will form the base step in the 
induction argument of Corollary~\ref{thm:last}. 
Note that a triangle $\triangle$ is in $D$ if the local cut of $\cutt_1$ in 
$\triangle$ changes boundary components when we consider $\cutt_2$.  The goal in 
each lemma is to focus on one triangle $\triangle$ in $D$ and find a sequence of 
reflections that allows us to swap the the local cut in $\triangle$ from one 
boundary component to the other. 
\begin{lemma}\label{lem:0tri}
If $\triangle_1$ shares an edge with 
$\triangle_2$, then $\alg_1$ is derived equivalent to $\alg_2$.
\end{lemma}

\begin{proof}
Let $\te i$ be the edge shared between $\triangle_1$ and $\triangle_2$. The fourth 
and fifth reflections in the dictionary show us that there are always suitable 
reflections (or co-reflections) such that both local cuts are incident to $\te i$. 
Then $R_iR_i$ is a sequence of reflections that send either $(S,M,T)$ to $(S,M',T')$ 
or vice versa.
\end{proof}

\begin{lemma}\label{lem:1tri}
If there is exactly one triangle separating $\triangle_1$ and $\triangle_2$, then 
$\alg_1$ and $\alg_2$ are derived equivalent. 
\end{lemma}

\begin{proof}
In Figure~\ref{fig:1tri} we see the four possible arrangements of $\triangle_1$ 
and $\triangle_2$.
\begin{figure}
\centering
\begin{tikzpicture}
{ []
	\draw (0,0) -- (4,0) (0,-2) -- (4,-2);
	\foreach \x [count=\n] in {1,1.6,2.15,3.25}
		\node[vert,name=t\n] at (\x,0){};
	\foreach \x [count=\n] in {1.5,2.25,3}
		\node[vert,name=b\n] at (\x,-2){};
	\draw (t1) edge[bend right=35]  (t3) edge (b1) 
		(t3) edge node[int] {$i$} (b1)
		(b1) edge[bend left=35] (b3)  edge node[int] {$j$} (t4)
		(b3) edge (t4);
	\draw[rcut] 	
		($(t3) +(45:.25cm)$) -- ($(t3) +(-135:.5cm)$) 
		($(b1) +(40:.5cm)$) -- ($(b1) +(-140:.25cm)$);
	\draw[bcut] 	
		($(t4) +(65:.25cm)$) -- ($(t4) +(-115:.5cm)$) 
		($(b1) +(90:.5cm)$) -- ($(b1) +(-90:.25cm)$);
	\node at (0,-1) {(a)};
}
{ [xshift =5cm]
	\draw (0,0) -- (4,0) (0,-2) -- (4,-2);
	\foreach \x [count=\n] in {1,1.6,2.15,2.85,3.5}
		\node[vert,name=t\n] at (\x,0){};
	\foreach \x [count=\n] in {1.5,2.85}
		\node[vert,name=b\n] at (\x,-2){};
	\draw (t1) edge[bend right=35] (t3) edge (b1) 
		(b1) edge node[int] {$i$} (t3) 
		(b2) edge node[int] {$j$} (t3)  edge (t5)
		(t3) edge[bend right=35] (t5);
	\draw[rcut] 	
		($(t3) +(45:.25cm)$) -- ($(t3) +(-135:.5cm)$) 
		($(b2) +(90:.5cm)$) -- ($(b2) +(-90:.25cm)$);
	\draw[bcut] 	
		($(t3) +(135:.25cm)$) -- ($(t3) +(-45:.5cm)$) 
		($(b1) +(90:.5cm)$) -- ($(b1) +(-90:.25cm)$);
	\node at (0,-1) {(b)};
}
{ [yshift=-3cm]
	\draw (0,0) -- (4,0) (0,-2) -- (4,-2);
	\foreach \x [count=\n] in {1.5,2.25,3}
		\node[vert,name=t\n] at (\x,0){};
	\foreach \x [count=\n] in {1,1.6,2.15,3.25}
		\node[vert,name=b\n] at (\x,-2){};
	\draw (b1) edge[bend left=35] (b3) edge (t1) 
		(b3) edge node[int] {$i$} (t1)
		(t1) edge[bend right=35] (t3)  edge node[int] {$j$} (b4)
		(t3) edge (b4);
	\draw[rcut] 	
		($(b4) +(115:.5cm)$) -- ($(b4) +(-65:.25cm)$) 
		($(t1) +(-90:.5cm)$) -- ($(t1) +(90:.25cm)$);
	\draw[bcut] 	
		($(b3) +(125:.5cm)$) -- ($(b3) +(-55:.25cm)$) 
		($(t1) +(-40:.5cm)$) -- ($(t1) +(140:.25cm)$);
	\node at (0,-1) {(c)};
}
{ [xshift =5cm,yshift=-3cm]
	\draw (0,0) -- (4,0) (0,-2) -- (4,-2);
	\foreach \x [count=\n] in {1.5,2.85} 
		\node[vert,name=t\n] at (\x,0){};
	\foreach \x [count=\n] in {1,1.6,2.15,2.85,3.5}
		\node[vert,name=b\n] at (\x,-2){};
	\draw (b1) edge[bend left=35] (b3) edge (t1) 
		(t1) edge node[int] {$j$} (b3) 
		(t2) edge node[int] {$i$} (b3)  edge (b5)
		(b3) edge[bend left=35] (b5);
	\draw[rcut] 	
		($(b3) +(-125:.25cm)$) -- ($(b3) +(55:.5cm)$) 
		($(t1) +(-90:.5cm)$) -- ($(t1) +(90:.25cm)$);
	\draw[bcut] 	
		($(b3) +(-55:.25cm)$) -- ($(b3) +(125:.5cm)$) 
		($(t2) +(-90:.5cm)$) -- ($(t2) +(90:.25cm)$);
	\node at (0,-1) {(d)};
}
\end{tikzpicture}
\caption{The possible arrangements of $\triangle_1$ and 
$\triangle_2$ in Lemma~\ref{lem:1tri}.  The red lines 
represent $\cutt_1$, the blue dashed lines $\cutt_2$.}\label{fig:1tri}
\end{figure}
In each case we reduce to Lemma~\ref{lem:0tri} by a reflection at $i$. Specifically, the
desired sequence of reflections is $R_iR_jR_iR_j$. Note that in each of these pictures we have
assumed that the left most triangle was always cut along the upper boundary component. By flipping
each picture along the horizontal axis, we can see each situation with the left most triangle cut in
the lower boundary component. In these cases the desired reduction come from the co-reflection at
$j$.
\end{proof}

\begin{lemma}\label{lem:2tri}
If there are exactly two triangles separating $\triangle_1$ and $\triangle_2$, then 
$\alg_1$ and $\alg_2$ are derived equivalent. 
\end{lemma}

\begin{proof}
We proceed as in Lemma~\ref{lem:1tri}.  The possible configurations for $\triangle_1$ and $\triangle_2$ 
are shown in Figure~\ref{fig:2tri}. As in Lemma~\ref{lem:1tri}, we focus on those cases were the 
left most triangle is cut in the upper boundary.
\begin{figure}
\centering
\begin{tikzpicture}
{ []
	\draw (0,0) -- (5,0) (0,-2) -- (5,-2);
	\foreach \x [count=\n] in {1,1.6,2.25}
		\node[vert,name=t\n] at (\x,0){};
	\foreach \x [count=\n] in {1.5,2.25,3,3.6,4.2}
		\node[vert,name=b\n] at (\x,-2){};
	\draw (t1) edge[bend right=35] node[int] {$m$} (t3) edge (b1) 
		(t3) edge node[int] {$i$} (b1) edge node[int] {$j$} (b2) edge node[int] {$k$} (b3)
		(b3) edge[bend left=35] node[int] {$\ell$} (b5)
		(b5) edge (t3);
	\draw[rcut] 	
		($(t3) +(45:.25cm)$) -- ($(t3) +(-135:.5cm)$) 
		($(b3) +(90:.5cm)$) -- ($(b3) +(-90:.25cm)$);
	\draw[bcut] 	
		($(t3) +(135:.25cm)$) -- ($(t3) +(-55:.5cm)$) 
		($(b1) +(90:.5cm)$) -- ($(b1) +(-90:.25cm)$);
	\node at (0,-1) {(a)};
}
{ [xshift =6cm]
	\draw (0,0) -- (5,0) (0,-2) -- (5,-2);
	\foreach \x [count=\n] in {1,1.6,2.25,3,3.75}
		\node[vert,name=t\n] at (\x,0){};
	\foreach \x [count=\n] in {2,2.75,3.5}
		\node[vert,name=b\n] at (\x,-2){};
	\draw (t1) edge[bend right=35] node[int] {$m$} (t3) edge (b1) 
		(b1) edge node[int] {$j$} (t4) edge node[int] {$k$} (t5) edge node[int,left] {$i$} (t3) 
		(b1) edge[bend left=35]  node[int] {$\ell$}  (b3) 
		(b3) edge (t5);
	\draw[rcut] 	
		($(t3) +(45:.25cm)$) -- ($(t3) +(-135:.5cm)$) 
		($(b1) +(40:.5cm)$) -- ($(b1) +(-140:.25cm)$);
	\draw[bcut] 	
		($(t5) +(65:.25cm)$) -- ($(t5) +(245:.5cm)$) 
		($(b1) +(100:.5cm)$) -- ($(b1) +(280:.25cm)$);
	\node at (0,-1) {(b)};
}
{ [yshift=-3cm]
	\draw (0,0) -- (5,0) (0,-2) -- (5,-2);
	\foreach \x [count=\n] in {.75,1.5,2.15,3}
		\node[vert,name=t\n] at (\x,0){};
	\foreach \x [count=\n] in {1.5,2.25,3,3.75}
		\node[vert,name=b\n] at (\x,-2){};
	\draw (t1) edge[bend right=35] node[int] {$m$} (t3) edge (b1) 
		(t3) edge node[int] {$i$} (b1) edge node[int] {$j$} (b2)
		(b2) edge node[int] {$k$} (t4)
		(b2) edge[bend left=35] node[int] {$\ell$} (b4) 
		(b4) edge (t4);
	\draw[rcut] 	
		($(t3) +(45:.25cm)$) -- ($(t3) +(-135:.5cm)$) 
		($(b2) +(45:.5cm)$) -- ($(b2) +(-135:.25cm)$);
	\draw[bcut] 	
		($(t4) +(90:.25cm)$) -- ($(t4) +(-90:.5cm)$) 
		($(b1) +(90:.5cm)$) -- ($(b1) +(90+180:.25cm)$);
	\node at (0,-1) {(c)};
}
{ [xshift =6cm,yshift=-3cm]
	\draw (0,0) -- (5,0) (0,-2) -- (5,-2);
	\foreach \x [count=\n] in {.75,1.5,2.15,3}
		\node[vert,name=t\n] at (\x,0){};
	\foreach \x [count=\n] in {1.5,2.25,3,3.75}
		\node[vert,name=b\n] at (\x,-2){};
	\draw (t1) edge[bend right=35] node[int] {$m$} (t3) edge (b1) 
		(t3) edge node[int] {$i$} (b1) 
		(b1) edge node[int] {$j$} (t4)
		(b2) edge node[int] {$k$} (t4)
		(b2) edge[bend left=35] node[int] {$\ell$} (b4) 
		(b4) edge (t4);
	\draw[rcut] 	
		($(t3) +(45:.25cm)$) -- ($(t3) +(-135:.5cm)$) 
		($(b2) +(45:.5cm)$) -- ($(b2) +(-135:.25cm)$);
	\draw[bcut] 	
		($(t4) +(90:.25cm)$) -- ($(t4) +(90+180:.5cm)$) 
		($(b1) +(90:.5cm)$) -- ($(b1) +(90+180:.25cm)$);
	\node at (0,-1) {(d)};
}
\end{tikzpicture}
\caption{The possible arrangements of $\triangle_1$ and $\triangle_2$ in 
Lemma~\ref{lem:2tri} The red lines 
represent $\cutt_1$, the blue dashed lines $\cutt_2$.}\label{fig:2tri}
\end{figure}
First note that case (d) reduces to (c) by a reflection at $j$. A reflection at $i$ then reduces (c) to 
Lemma~\ref{lem:0tri}.  Let $R^*$ denote the corresponding sequence of reflections from Lemma~\ref{lem:0tri}.  Then the desired sequence of reflections in case (c) is $R_iR^*R_i$.

Similarly, (a) and (b) reduce to Lemma~\ref{lem:0tri} by a reflection at $j$.
By reflecting at $j$ we introduce a new cut triangle connecting $\triangle_1$ and $\triangle_2$. The cut will be in the lower boundary and upper boundary for case (a) and (b) respectively. We will explicitly discuss the sequence of reflections in case (a), the reflections for case (b) can be
found in a similar manner. We may then apply Lemma~\ref{lem:0tri} to this new triangle and
$\triangle_1$, so as to move the cut in the upper boundary to the lower. We then apply
Lemma~\ref{lem:0tri} to the new middle triangle and $\triangle_2$, to move the cut in $\triangle_2$ 
to the upper boundary.
\end{proof}

Using the above lemmas we get the following special case of Theorem~\ref{mainthm}.

\begin{cor}\label{thm:last}
Let $S$ be an annulus and $\alg_1$ and $\alg_2$ be algebras coming from 
$\cutt_1$ $\cutt_2$ respectively. If $\cutt_1(\bcpt) = \cutt_{2}(\bcpt)$ 
for both boundary components $\bcpt$ in $S$, that is $(\cutt_1,\cutt_2)$ 
are equi-distributed, then $\alg_1$ is derived equivalent to $\alg_2$.
\end{cor}

\begin{proof} 
Let $D=\{ \triangle : \cutt_{1,\triangle}(\bcpt)\neq \cutt_{2,\triangle}(\bcpt)\}$, 
we begin by assuming that $\#D=2$, say $D=\{\triangle_1,\triangle_2\}$. In this setup we may
even assume that there are no internal triangles separating $\triangle_1$ and $\triangle_2$, the process
we will describe is transitive between internal triangles. We will show, by induction, that there is
a sequence of reflections that allow us to swap the cuts in $\triangle_1$ and $\triangle_2$.

Throughout we will denote cuts as in the dictionary, by red lines bisecting the cut vertex between
the endpoints of the resulting relation. Let $\triangle_1$ be the triangle containing $m$ and $i$. We
focus on the different configurations for $\triangle_1$, the different cases corresponding to
different configurations of $\triangle_2$ are hidden and dealt with in the induction step.

The initial case. By Lemmas~\ref{lem:0tri}, \ref{lem:1tri}, \ref{lem:2tri} we can resolve $\triangle_1$
and $\triangle_2$ when there are $0$, $1$ or $2$ triangles separating them. Now assume that we can
resolve $\triangle_1$ and $\triangle_2$ with up to $t$ triangles separating them. Let $R^*$ denote the
composition of reflections necessary for the induction hypothesis. Then we have one of the following
picture for $(S,M,T)$:
 
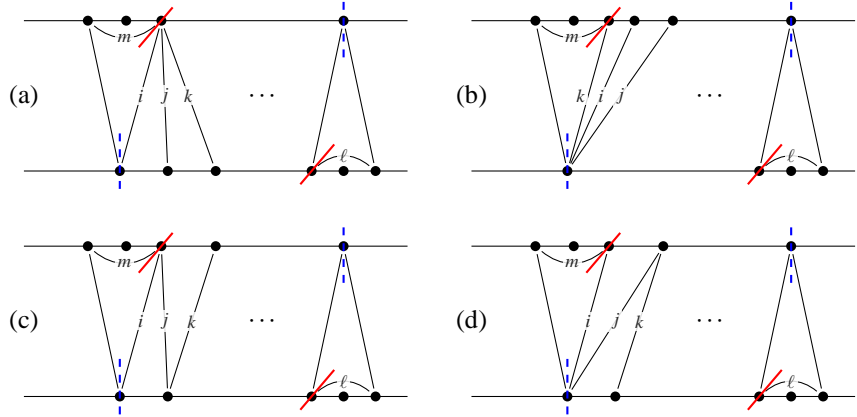
\[\begin{tikzpicture}[xscale=.85]

{ []
	\draw (0,0) -- (6,0) (0,-2) -- (6,-2);
	\foreach \x [count=\n] in {1,1.6,2.15,5}
		\node[vert,name=t\n] at (\x,0){};
	\foreach \x [count=\n] in {1.5,2.25,3,4.5,5,5.5}
		\node[vert,name=b\n] at (\x,-2){};
	\draw (t1) edge[bend right=35] node[int] {$m$} (t3) edge (b1) 
		(t3) edge node[int] {$i$} (b1) edge node[int] {$j$} (b2) edge node[int] {$k$} (b3)
		(b4) edge[bend left=35] node[int] {$\ell$} (b6)  edge (t4)
		(b6) edge (t4);
	\draw[rcut] 	
		($(t3) +(45:.25cm)$) -- ($(t3) +(-135:.5cm)$) 
		($(b4) +(45:.5cm)$) -- ($(b4) +(-135:.25cm)$);
	\draw[bcut] 	
		($(t4) +(90:.25cm)$) -- ($(t4) +(90+180:.5cm)$) 
		($(b1) +(90:.5cm)$) -- ($(b1) +(90+180:.25cm)$);
	\node at (3.75,-1) {$\cdots$};
	\node at (0,-1) {(a)};
}
{ [xshift =7cm]
	\draw (0,0) -- (6,0) (0,-2) -- (6,-2);
	\foreach \x [count=\n] in {1,1.6,2.15,2.55,3.15,5}
		\node[vert,name=t\n] at (\x,0){};
	\foreach \x [count=\n] in {1.5,4.5,5,5.5}
		\node[vert,name=b\n] at (\x,-2){};
	\draw (t1) edge[bend right=35] node[int] {$m$} (t3) edge (b1) 
		(b1) edge node[int] {$i$} (t4) edge node[int] {$j$} (t5) edge node[int,left] {$k$} (t3) 
		(b2) edge[bend left=35]  node[int] {$\ell$}  (b4)  edge (t6)
		(b4) edge (t6);
	\draw[rcut] 	
		($(t3) +(45:.25cm)$) -- ($(t3) +(-135:.5cm)$) 
		($(b2) +(45:.5cm)$) -- ($(b2) +(-135:.25cm)$);
	\draw[bcut] 	
		($(t6) +(90:.25cm)$) -- ($(t6) +(90+180:.5cm)$) 
		($(b1) +(90:.5cm)$) -- ($(b1) +(90+180:.25cm)$);
	\node at (3.75,-1) {$\cdots$};
	\node at (0,-1) {(b)};
}
{ [yshift=-3cm]
	\draw (0,0) -- (6,0) (0,-2) -- (6,-2);
	\foreach \x [count=\n] in {1,1.6,2.15,3,5}
		\node[vert,name=t\n] at (\x,0){};
	\foreach \x [count=\n] in {1.5,2.25,4.5,5,5.5}
		\node[vert,name=b\n] at (\x,-2){};
	\draw (t1) edge[bend right=35] node[int] {$m$} (t3) edge (b1) 
		(t3) edge node[int] {$i$} (b1) edge node[int] {$j$} (b2)
		(b2) edge node[int] {$k$} (t4)
		(b3) edge[bend left=35] node[int] {$\ell$} (b5)  edge (t5)
		(b5) edge (t5);
	\draw[rcut] 	
		($(t3) +(45:.25cm)$) -- ($(t3) +(-135:.5cm)$) 
		($(b3) +(45:.5cm)$) -- ($(b3) +(-135:.25cm)$);
	\draw[bcut] 	
		($(t5) +(90:.25cm)$) -- ($(t5) +(90+180:.5cm)$) 
		($(b1) +(90:.5cm)$) -- ($(b1) +(90+180:.25cm)$);		
	\node at (3.75,-1) {$\cdots$};
	\node at (0,-1) {(c)};
}
{ [xshift =7cm,yshift=-3cm]
	\draw (0,0) -- (6,0) (0,-2) -- (6,-2);
	\foreach \x [count=\n] in {1,1.6,2.15,3,5}
		\node[vert,name=t\n] at (\x,0){};
	\foreach \x [count=\n] in {1.5,2.25,4.5,5,5.5}
		\node[vert,name=b\n] at (\x,-2){};
	\draw (t1) edge[bend right=35] node[int] {$m$} (t3) edge (b1) 
		(t3) edge node[int] {$i$} (b1) 
		(t4) edge node[int] {$j$} (b1)
		(b2) edge node[int] {$k$} (t4)
		(b3) edge[bend left=35] node[int] {$\ell$} (b5)  edge (t5)
		(b5) edge (t5);
	\draw[rcut] 	
		($(t3) +(45:.25cm)$) -- ($(t3) +(-135:.5cm)$) 
		($(b3) +(45:.5cm)$) -- ($(b3) +(-135:.25cm)$);
	\draw[bcut] 	
		($(t5) +(90:.25cm)$) -- ($(t5) +(90+180:.5cm)$) 
		($(b1) +(90:.5cm)$) -- ($(b1) +(90+180:.25cm)$);
	\node at (3.75,-1) {$\cdots$};
	\node at (0,-1) {(d)};
}
\end{tikzpicture}\]
Note that case (d) reduces to (c) by a reflection at $j$, hence we only focus on
(a), (b), and (c). The desired sequence of reflections is $R_jR^*R_iR_iR_j$,
$R_jR_iR_iR^*R_j$, $R_iR^*R_i$ for (a), (b) and (c) respectively. For example,
using the dictionary we get the following sequence of pictures in
Figure~\ref{fig:mainThm}.
\begin{figure}
\centering 
 
\begin{tikzpicture}[yscale=.7]
{ [yscale=.85]
\draw (0,0) -- (6,0) (0,-2) -- (6,-2);
\foreach \x [count=\n] in {1,1.6,2.15,5}
	\node[vert,name=t\n] at (\x,0){};
\foreach \x [count=\n] in {1.5,2.25,3,4.5,5,5.5}
	\node[vert,name=b\n] at (\x,-2){};
\draw (t1) edge[bend right=35] node[int] {$m$} (t3) edge (b1) 
	(t3) edge node[int] {$i$} (b1) edge node[int] {$j$} (b2) edge node[int] {$k$} (b3)
	(b4) edge[bend left=35] node[int] {$\ell$} (b6)  edge (t4)
	(b6) edge (t4);
\draw[rcut] 	
	($(t3) +(45:.25cm)$) -- ($(t3) +(-135:.5cm)$) 
	($(b4) +(45:.5cm)$) -- ($(b4) +(-135:.25cm)$);
\draw[bcut] 	
	($(t4) +(90:.25cm)$) -- ($(t4) +(90+180:.5cm)$) 
	($(b1) +(90:.5cm)$) -- ($(b1) +(90+180:.25cm)$);
\node at (3.75,-1) {$\cdots$};
}
\draw[->] (3,-2.5) -- (3,-3.5) node[right,pos=.5] {$R_j$};
{ [yshift=-4cm,yscale=.85]
\draw (0,0) -- (6,0) (0,-2) -- (6,-2);
\foreach \x [count=\n] in {1,1.6,2.15,5}
	\node[vert,name=t\n] at (\x,0){};
\foreach \x [count=\n] in {1.5,2.25,3,4.5,5,5.5}
	\node[vert,name=b\n] at (\x,-2){};
\draw (t1) edge[bend right=35] node[int] {$m$} (t3) edge (b1) 
	(t3) edge node[int] {$i$} (b1)  edge node[int] {$k$} (b3)
	(b1) edge[bend left=35] node[int] {$j$} (b3)
	(b4) edge[bend left=35] (b6)  edge (t4)
	(b6) edge (t4);
\draw[rcut] 	
	($(t3) +(45:.25cm)$) -- ($(t3) +(-135:.5cm)$) 
	($(b4) +(45:.5cm)$) -- ($(b4) +(-135:.25cm)$)
	($(b1) +(45:.5cm)$) -- ($(b1) +(-135:.25cm)$);
\node at (3.75,-1) {$\cdots$};
}
\draw[->] (3,-6.5) -- (3,-7.5) node[right,pos=.5] {$R_iR_i$};
{ [yshift=-8cm,yscale=.85]
\draw (0,0) -- (6,0) (0,-2) -- (6,-2);
\foreach \x [count=\n] in {1,1.6,2.15,5}
	\node[vert,name=t\n] at (\x,0){};
\foreach \x [count=\n] in {1.5,2.25,3,4.5,5,5.5}
	\node[vert,name=b\n] at (\x,-2){};
\draw (t1) edge[bend right=35] node[int] {$m$} (t3) edge (b1) 
	(t3) edge node[int] {$i$} (b1)  edge node[int] {$k$} (b3)
	(b1) edge[bend left=35] node[int] {$j$} (b3)
	(b4) edge[bend left=35] (b6)  edge (t4)
	(b6) edge (t4);
\draw[rcut] 	
	($(t3) +(90:.25cm)$) -- ($(t3) +(-90:.5cm)$) 
	($(b4) +(45:.5cm)$) -- ($(b4) +(-135:.25cm)$)
	($(b1) +(90:.5cm)$) -- ($(b1) +(-90:.25cm)$);
\node at (3.75,-1) {$\cdots$};
}
\draw[->] (3,-10.5) -- (3,-11.5) node[right,pos=.5] {$R^*$};
{ [yshift=-12cm,yscale=.85]
\draw (0,0) -- (6,0) (0,-2) -- (6,-2);
\foreach \x [count=\n] in {1,1.6,2.15,5}
	\node[vert,name=t\n] at (\x,0){};
\foreach \x [count=\n] in {1.5,2.25,3,4.5,5,5.5}
	\node[vert,name=b\n] at (\x,-2){};
\draw (t1) edge[bend right=35] node[int] {$m$} (t3) edge (b1) 
	(t3) edge node[int] {$i$} (b1)  edge node[int] {$k$} (b3)
	(b1) edge[bend left=35] node[int] {$j$} (b3)
	(b4) edge[bend left=35] (b6)  edge (t4)
	(b6) edge (t4);
\draw[rcut] 	
	($(t4) +(90:.25cm)$) -- ($(t4) +(-90:.5cm)$) 
	($(b3) +(135:.5cm)$) -- ($(b3) +(-45:.25cm)$)
	($(b1) +(90:.5cm)$) -- ($(b1) +(-90:.25cm)$);
\node at (3.75,-1) {$\cdots$};
}
\draw[->] (3,-14.5) -- (3,-15.5) node[right,pos=.5] {$R_j$};
{ [yshift=-16cm,yscale=.85]
\draw (0,0) -- (6,0) (0,-2) -- (6,-2);
\foreach \x [count=\n] in {1,1.6,2.15,5}
	\node[vert,name=t\n] at (\x,0){};
\foreach \x [count=\n] in {1.5,2.25,3,4.5,5,5.5}
	\node[vert,name=b\n] at (\x,-2){};
\draw (t1) edge[bend right=35] node[int] {$m$} (t3) edge (b1) 
	(t3) edge node[int] {$i$} (b1) edge node[int] {$j$} (b2)  edge node[int] {$k$} (b3)
	(b4) edge[bend left=35] (b6)  edge (t4)
	(b6) edge (t4);
\draw[rcut] 	
	($(t4) +(90:.25cm)$) -- ($(t4) +(-90:.5cm)$) 
	($(b1) +(90:.5cm)$) -- ($(b1) +(-90:.25cm)$);
\node at (3.75,-1) {$\cdots$};
}
\end{tikzpicture}
\caption{The sequence of reflections for case (a) of Corollary~\ref{thm:last}.}\label{fig:mainThm}
\end{figure}
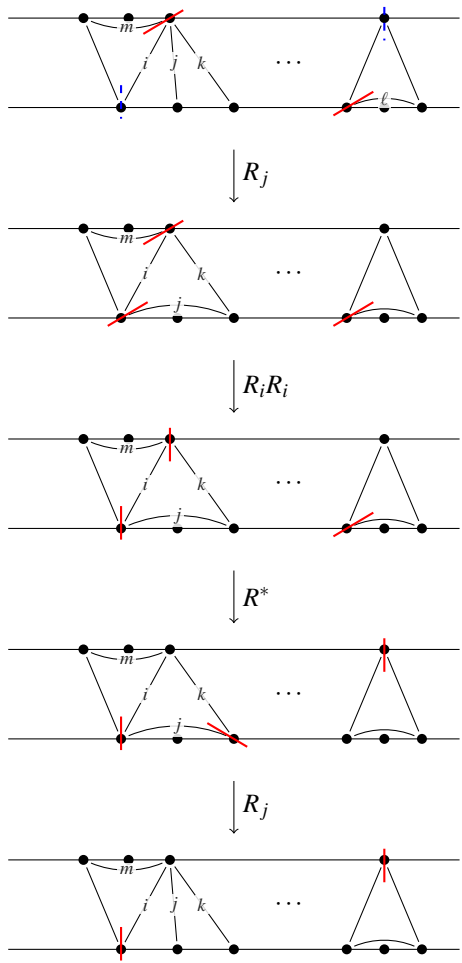
Note that if the cut incident to $m$ (resp. $\ell$) had been at the other vertex, a 
double (co)-reflection at $m$ (resp. $\ell$), would give us the above pictures. 

The proof for $\#D=2$ generalizes for arbitrary $\#D=2m$ by applying this 
proof to pairs $\triangle_1$ and $\triangle_2$ in $D$ with a minimal number of triangles separating them, 
doing so until all pairs have been resolved.  
\end{proof}



\end{document}